\documentclass[a4paper]{article}

\usepackage{amsfonts}
\usepackage{amsmath}
\usepackage{mathrsfs}
\usepackage{amssymb,color}

\usepackage[latin1]{inputenc}
\usepackage{amsthm}
\usepackage{graphicx}
\usepackage{enumerate}
\newtheorem{theorem}{Theorem}[section]

\newtheorem{definition}[theorem]{Definition}
\newtheorem{lemma}[theorem]{Lemma}
\newtheorem{remark}[theorem]{Remark}
\newtheorem{corollary}[theorem]{Corollary}

\newcommand{\R}{{{\mathbb R}}}

\begin{document}
	\title{\textbf{Constant sign Green's function for simply supported beam equation.\footnote{Partially supported by Ministerio de Educaci\'on, Cultura y Deporte, Spain and FEDER, project MTM2013-43014-P.} }}
	\date{}
	\author{Alberto Cabada   \;and  Lorena Saavedra\footnote{Supported by Plan I2C scholarship, Conseller\' {\i}a de Educaci\'on, Cultura e O.U., Xunta de Galicia, and FPU scholarship, Ministerio de Educaci\'on, Cultura y Deporte, Spain.}\\Departamento de Análise Matemática,\\ Facultade de  Matemáticas,\\Universidade Santiago de Compostela,\\ Santiago de Compostela, Galicia,
		Spain\\
		alberto.cabada@usc.es, lorena.saavedra@usc.es
	} 
	\maketitle
	\begin{abstract}

The aim of this paper consists on the study of the following fourth-order operator:
\begin{equation}\label{Ec::T4}
T[M]\,u(t)\equiv u^{(4)}(t)+p_1(t)\,u'''(t)+p_2(t)\,u''(t)+M\,u(t)\,,\ t\in I \equiv [a,b]\,,
\end{equation}
coupled with the two point boundary conditions:
\begin{equation}\label{Ec::cf}
	u(a)=u(b)=u''(a)=u''(b)=0\,.	
\end{equation}

So, we define the following space:
	\begin{equation}\label{Ec::esp}
	X=\left\lbrace u\in C^4(I)\quad\mid\quad u\text{ satisfies boundary conditions \eqref{Ec::cf}}\right\rbrace \,.	
	\end{equation}

	Here $p_1\in C^3(I)$ and $p_2\in C^2(I)$.
	
By assuming that the second order linear differential equation 
	\begin{equation}\label{Ec::2or}
	L_2\, u(t)\equiv u''(t)+p_1(t)\,u'(t)+p_2(t)\,u(t)=0\,,\quad t\in I,
	\end{equation}
	is disconjugate on $I$, we  characterize  the parameter's set where the Green's function related to operator $T[M]$ in $X$ is of constant sign on $I \times I$. Such characterization is equivalent to the strongly inverse positive (negative) character of operator $T[M]$ on $X$ and comes from the first eigenvalues of operator $T[0]$ on suitable spaces.
		\end{abstract}
\section{Introduction}
In this paper, we  characterize the strongly inverse positive and negative character of  operator 
\[T[M]\,u(t)\equiv u^{(4)}(t)+p_1(t)\,u'''(t)+p_2(t)\,u''(t)+M\,u(t)\,,\quad t\in I \equiv [a,b]\,,\]
on the space of definition
\[	X=\left\lbrace u\in C^4(I)\quad\mid\quad u(a)=u(b)=u''(a)=u''(b)=0\right\rbrace \,,\]
which corresponds to the simply supported beam boundary conditions.

Once  we have obtained such result about this operator, we would be  able to obtain additional sufficient conditions which ensure the strongly inverse negative or positive character of operator
\[T[M]\,u(t)\equiv u^{(4)}(t)+p_1(t)\,u'''(t)+p_2(t)\,u''(t)+c(t)\,u(t)\,,\quad t\in I \,,\]
for a given continuous function, $c$.

This problem has been studied along the time. The particular case where $p_1(t)\equiv p_2(t)\equiv0$ on $I$ has been considered in many papers: In \cite{Sch} the result for $p_1=p_2=0$ here obtained is proved for the strongly inverse positive character. The  strongly inverse negative character for that case has been proved in \cite{cacisa}. It is important to mention that in both cases the expression of the Green's function has been used. Moreover, no spectral relationship with operator $u^{(4)}$ has been found in both references.

Moreover, in \cite{Dra}, weaker sufficient conditions to ensure either the strongly inverse positive or negative character are given for operator $u^{(4)}+c(t)\,u$. In \cite{BaFry}, it is studied the operator $u^{(4)}-(\alpha^2+\beta^2)u''+\alpha^2\,\beta^2\,u$ defined in a complex domain, with $\alpha^2\neq\beta^2$. In this case, some sufficient conditions to ensure the inverse positive character are obtained. In \cite{Liu},  some results which ensure the existence of one or more positive solutions of the problem $u^{(4)}(t)-f(t,u(t),u''(t))=0$ with the boundary conditions \eqref{Ec::cf} on the interval $[0,1]$  are obtained.

Fourth order problems with different boundary conditions have also been studied. For instance, in \cite{CE} it is characterized the inverse positive character of operator $u^{(4)}$ coupled with the boundary conditions $u(a)=u'(a)=u(b)=u'(b)=0$, which corresponds to the clamped beam boundary conditions. Furthermore, inverse negative character of this operator with the boundary conditions $u(a)=u'(a)=u''(a)=u(b)=0$ and $u(a)=u(b)=u'(b)=u''(b)=0$ has been studied in \cite{CaFe}. 

In all these cases the expression of  the related Green's function was needed to characterize the inverse positive or inverse negative character of the operator. In \cite{CaSaa}, without knowing the expression of Green's function it is obtained a characterization of inverse positive or inverse negative character for general $n^{\rm th}$ - order operators defined in the following spaces of definition:
 \begin{equation}\label{Ec::Xk}X_k=\left\lbrace u\in C^{n}(I)\ \mid\ u(a)=\cdots=u^{(k-1)}(a)=u(b)=\cdots=u^{(n-k-1)}(b)=0\right\rbrace \,,\end{equation}
 where $1\leq k\leq n-1$.
 
 Such characterization follows from spectral theory on suitable spaces related to the considered $n^{\rm th}$ - order operator.
 
 In this paper we are going to follow some ideas of this result to deduce the characterization of strongly inverse positive (negative) character of \eqref{Ec::T4}-\eqref{Ec::cf}.
 
\section{Preliminaries}

In this section, for the convenience of the reader, we introduce the fundamental tools in the theory of  Green's functions that will be used in the development of further sections. Some of these results can be found in \cite[Chapter 3]{Cop} and are valid for the general $n^{\rm th}-$ order linear operator
\begin{equation}\label{Op::n}
L_n[M]\,u(t)\equiv u^{(n)}(t)+p_1(t)\,u^{(n-1)}(t)+\cdots+ p_{n-1}(t)\,u'(t)+(p_n(t)+M)\,u(t)\,,
\end{equation} 
with $t \in I $ and $p_j\in C^{n-j}(I)$, $j=1, \ldots,n$.

\begin{definition}
 The $n^{\rm th}-$ order linear differential equation 
\begin{equation}\label{Ec::n}
L_n[M]\,u(t)=0\,,\quad t \in I
\end{equation}
is said to be disconjugate on  $I$ if every non trivial solution has less than $n$ zeros on $I$, multiple zeros being counted according to their multiplicity.
\end{definition}
\begin{definition}
	The functions $u_1,\dots, u_n \in C^n(I)$ are said to form a Markov system on  $I$ if the $n$ Wronskians
	\begin{equation}
	W(u_1,\dots,u_j)=\left| \begin{array}{ccc}
	u_1&\cdots&  u_j\\
	\vdots&\cdots&\vdots\\
	u_1^{(j-1)}&\cdots&u_j^{(j-1)}\end{array}\right| \,,\quad j=1,\dots,n \,,
	\end{equation}
	are positive throughout $I$.
\end{definition}

\begin{theorem}\label{T::4}
	The linear differential equation \eqref{Ec::n} has a Markov fundamental system of solutions on  $I$ if, and only if, it is disconjugate on $I$.
	
\end{theorem}

 \begin{theorem}\label{T::22}
 	The linear differential equation (\ref{Ec::n}) has a Markov system of solutions if, and only if, operator $L_n[M]$ has a representation of the form
 	\begin{equation}
 	\label{e-descomp}
 	L_n[M]\,u\equiv v_1 \,v_2\,\dots\,v_n \dfrac{d}{dt}\left( \dfrac{1}{v_n}\,\dfrac{d}{dt}\left( \cdots \dfrac{d}{dt}\left( \dfrac{1}{v_2} \dfrac{d}{dt}\left( \dfrac{1}{v_1}\,u\right) \right) \right) \right) \,,	
 	\end{equation}
 	where $v_k>0$  on $I$ and $v_k\in C^{n-k+1}(I)$ for all $k=1,\dots,n$.
 \end{theorem}

 	\begin{theorem}\label{T::comp}
 		Let $\hat L$ and $\tilde L$ be two $n^{\rm th}$ and $m^{\rm th}-$ order (respectively) linear differential operators following the expression \eqref{Op::n} for adequate coefficients $\hat p_k$ and $\tilde p_k$. If both equations $\hat L\,y=0$ and $\tilde L\,y=0$ are disconjugate on the interval $I$, then the composite $(n+m)^{\rm th}-$ order linear  equation $\hat L\,(\tilde L\,y)=0$ is also disconjugate on $I$.
 	\end{theorem}

 	The following result, which appears on \cite[Theorem 3.2]{Neh}, shows a property of the eigenvalues of a disconjugate operator in these particular spaces $X_k$.
 	
 	\begin{theorem}\label{T::10}
 		Let $\bar{M}\in\mathbb{R}$ be such that equation $L_n[\bar{M}]\,u(t)=0$ is disconjugate on $I$. Then for any $1\leq k \leq n-1$ the following properties hold:
 		\begin{itemize}
 			\item If $n-k$ is even, there is not any eigenvalue of $L_n[\bar{M}]$ on $X_k$  such that $\lambda<0$.
 			\item If $n-k$ is odd, there is not any eigenvalue of $L_n[\bar{M}]$ on $X_k$  such that $\lambda>0$.
 		\end{itemize} 
 	\end{theorem}
 
In order to introduce the concept of Green's function related to  operator $T[M]$ in $X$, we consider the following equivalent first order vectorial problem:

\begin{equation}\label{Ec::vec}
x'(t)=A(t)\, x(t)\,,\ t\in I\,,\quad 
B\,x(a)+C\,x(b)=0,
\end{equation}
with $x(t) \in \R^4$, $A(t), \, B,\,\ C\in \mathcal{M}_{4\times 4}$, defined by
\[x(t)=\left( \begin{array}{c} 
u(t)\\
u'(t)\\
u''(t)\\
u'''(t) \end{array}\right),\, \quad A(t)=\left( \begin{array}{cccc}
0&1&0&0\\0&0&1&0\\0&0&0&1\\-M&0&-p_2(t)&-p_1(t) \end{array}\right), \]

\begin{equation}\label{Ec::Cf}
B=\left( \begin{array}{cccc} 1&0&0&0\\0&0&1&0\\0&0&0&0\\0&0&0&0\end{array}\right), \;\quad C=\left( \begin{array}{cccc} 0&0&0&0\\0&0&0&0\\
1&0&0&0\\0&0&1&0\end{array}\right).
\end{equation}

\begin{definition} \label{Def::G}
	We say that $G$ is a Green's function for vectorial problem \eqref{Ec::vec} if it satisfies the following properties:
	\begin{itemize}
		\item[$\mathrm{(G1)}$] $G\equiv (G_{i,j})_{i,j\in\{1,\dots,4\}}\colon (I\times I)\backslash \left\lbrace (t,t)\,,\ t\in I\right\rbrace \rightarrow \mathcal{M}_{4\times 4}$.
		
		\item[$\mathrm{(G2)}$] $G$ is a $C^{1}$ function on the triangles $\left\lbrace (t,s)\in \mathbb{R}^2\,,\quad a\leq s<t\leq b\right\rbrace $ and $\left\lbrace (t,s)\in \mathbb{R}^2\,,\ a\leq t < s\leq b\right\rbrace $.
		
		\item[$\mathrm{(G3)}$] For all $i\neq j$ the scalar functions $G_{i,j}$ have a continuous extension to $I\times I$.
		
		\item[$\mathrm{(G4)}$] For all $s\in(a,b)$, the following equality holds:
		\[\dfrac{\partial }{\partial t}\, G(t,s)=A(t)\,G(t,s)\,,\quad \text{for all } t\in I\backslash \left\lbrace s\right\rbrace \,.\]
		
		\item[$\mathrm{(G5)}$] For all $s\in(a,b)$ and $i\in\left\lbrace 1,\dots, 4\right\rbrace $, the following equalities are fulfilled:
		\[\lim_{t\rightarrow s^-}G_{i,i}(s,t)=\lim_{t\rightarrow s^+}G_{i,i}(t,s)=1+\lim_{t\rightarrow s^-}G_{i,i}(t,s)=1+\lim_{t\rightarrow s^+}G_{i,i}(s,t)\,.\]
		
		\item[$\mathrm{(G6)}$] For each $s\in(a,b)$, the function $t\rightarrow G(t,s)$ satisfies the boundary conditions
		\[B\,G(a,s)+C\,G(b,s)=0\,.\]
	\end{itemize}
\end{definition}

\begin{remark}\label{Rm::2.5}
	On  previous definition, item $\mathrm{(G5)}$ can be modified to obtain the characterization of the lateral limits for $s=a$ and $s=b$ as follows:	
	\[\lim_{t\rightarrow a^+}G_{i,i}(t,a)=1+\lim_{t\rightarrow a^+}G_{i,i}(a,t)\,,\quad\text{and}\quad \lim_{t\rightarrow b^-}G_{i,i}(b,t)=1+\lim_{t\rightarrow b^-}G_{i,i}(t,b)\,.\]
\end{remark}

It is very well known that Green's function related to this problem follows the expression (\cite[Section 1.4]{Cab})
\begin{equation}\label{Ec:MG1} G(t,s)=\left( \begin{array}{cccc}
g_1(t,s)&g_2(t,s)&g_{3}(t,s)&g_M(t,s)\\&&&\\
\dfrac{\partial }{\partial t}\,g_1(t,s)& \dfrac{\partial }{\partial t}\,g_2(t,s)&\dfrac{\partial }{\partial t}\,g_{3}(t,s)& \dfrac{\partial }{\partial t}\,g_M(t,s)\\&&&\\
\dfrac{\partial^2 }{\partial t^2}\,g_1(t,s)& \dfrac{\partial^2 }{\partial t^2}\,g_2(t,s)&\dfrac{\partial^2 }{\partial t^2}\,g_{3}(t,s)& \dfrac{\partial^2 }{\partial t^2}\,g_M(t,s)\\&&&\\
\dfrac{\partial^3 }{\partial t^3}\,g_1(t,s)& \dfrac{\partial ^3}{\partial t^3}\,g_2(t,s)&\dfrac{\partial^3 }{\partial t^3}\,g_{3}(t,s)& \dfrac{\partial^3 }{\partial t^3}\,g_M(t,s)\end{array} \right) \,,\end{equation}
where $g_M$ is the scalar Green's function related to  operator $T[M]$ in $X$. 

Using  Definition \ref{Def::G} we can deduce the properties fulfilled by $g_M$. In particular, $g_M\in C^{2}(I\times I)$. Moreover  it is a $C^{4}-$ function on the triangles $\left\lbrace (t,s)\in \mathbb{R}^2\,,\quad a\leq s<t\leq b\right\rbrace $ and $\left\lbrace (t,s)\in \mathbb{R}^2\,,\ a\leq t < s\leq b\right\rbrace $, it satisfies, as a function of $t$, the two-point boundary value conditions \eqref{Ec::cf} and solves equation \eqref{Ec::T4} for all $t \in I \backslash \{s\}$.

Studying the matrix Green's function, we can make a relation between $g_1$, $g_2$, $g_3$ and $g_M$ and we can express them as follows in this particular case, (see \cite{CaSaa} for details),
\begin{eqnarray}\nonumber g_1(t,s)&=&-\dfrac{\partial^3}{\partial s^3}g_M(t,s)+p_1(s)\,\dfrac{\partial^2}{\partial s^2}g_M(t,s)+(2\,p_1'(s)-p_2(s))\,\dfrac{\partial}{\partial s}g_M(t,s)\\\nonumber\\\nonumber&&+(p_1''(s)-p_2'(s))\,g_M(t,s)\,,\\\nonumber
\\\nonumber
g_2(t,s)&=&\dfrac{\partial^2}{\partial s^2}g_M(t,s)-p_1(s)\,\dfrac{\partial}{\partial s}g_M(t,s)+(p_2(s)-p_1'(s))\,g_M(t,s)\,,\\\nonumber\\\label{Ec::g3}
g_3(t,s)&=& -\dfrac{\partial }{\partial s}g_M(t,s)+p_1(s)\,g_M(t,s)\,.
\end{eqnarray}

	Next result appears in \cite[Chapter 3, Theorem 9]{Cop}
	\begin{theorem}\label{T::ad}
		A linear differential equation \eqref{Ec::n} is disconjugate on  $I$ if, and only if, its adjoint equation, $L_n^*[M]\,y(t)=0$ is disconjugate on $I$.
	\end{theorem}
	
	We denote $g_M^*(t,s)$ as the Green's function related to the adjoint operator $L_n^*[M]$.
	
	In \cite[Section 1.4]{Cab} it is proved the following relationship
	\begin{equation} \label{Ec::gg}
	g^*_M(t,s)=g_M(s,t)\,,\quad\forall (t,s)\in I\times I\,.
	\end{equation}

	Now, we introduce the following space of functions:
		$$X_U=\left\{ u \in C^n(I), \quad 
 \sum_{j=0}^{n-1}\left(\alpha_j^i\, u^{(j)}(a) + \beta_j^i\,u^{(j)}(b)\right)=0,\; i=1,\ldots,n\right\},$$
being $\alpha_j^i,\beta_j^i$ real constants for all $i=1,\ldots,n,$ and $j=0,\ldots,n-1$.

	\begin{definition}
		\label{d-IP}
		Operator $L_n[M]$ is said to be inverse positive (inverse negative) on  $X_U$, if every function $u \in X_U$ such that $L_n[M]\, u \ge 0$ in $I$, satisfies $u\geq 0$ ($u\leq 0$) on $I$.
	\end{definition}
	
	Next results are proved in \cite[Sections 1.6 and 1.8]{Cab}.
	
	\begin{theorem}\label{T::in1}
		Operator $L_n[M]$ is inverse positive (inverse negative) on $X_U$ if, and only if, Green's function related to operator $L_n[M]$ in $X_U$ is non-negative (non-positive) on $I\times I$.
	\end{theorem}

	\begin{theorem}\label{T::d1}
		Let $M_1$, $M_2\in\mathbb{R}$ and suppose that operators $L_n[M_j]$, $j=1,2$, are invertible in $X_U$.
		Let $g_j$, $j=1,2$, be Green's functions related to  operators $L_n[M_j]$ and suppose that both functions have the same constant sign on $I \times I$. Then, if $M_1<M_2$, it is satisfied that $g_2\leq g_1$ on $I \times I$.
	\end{theorem}
	
	\begin{theorem}\label{T::int}
		Let $M_1<\bar{M}<M_2$ be three real constants. Suppose that operator $L_n[M]$ is invertible in $X_U$ for $M=M_j$, $j=1,2$ and that the corresponding Green's function satisfies $g_2\leq g_1\leq 0$ (resp. $0\leq g_2\leq g_1$) on $I\times I$. Then the operator $L_n[\bar{M}]$ is invertible in $X_U$ and the related Green's function $\bar{g}$ satisfies $g_2\leq \bar{g}\leq g_1\leq 0$ ($0\leq g_2\leq \bar{g}\leq g_1$) on $I\times I$.
	\end{theorem}
	
		We introduce a definition to our particular problem \eqref{Ec::T4} in the space $X$.
		\begin{definition}
			Operator $T[M]$ is said to be strongly inverse positive (strongly inverse negative) in $X$, if every function $u\in X$ such that $T[M]\,u\gneqq0$ in $I$, satisfies $u>0$ ($u<0$) on $(a,b)$ and, moreover $u'(a)>0$ and $u'(b)<0$ ($u'(a)<0$ and $u'(b)>0$).
		\end{definition}
		
	Next result shows a relationship between the Green's function's sign and the previous definition. The proof is an adaption to this situation of \cite[Corollaries 1.6.6 and 1.6.12]{Cab}
		
		\begin{theorem}\label{T::14}
			 Green's function related to  operator $T[M]$ in $X$ is positive (negative) a.e on $(a,b)\times (a,b)$ and, moreover, $\frac{\partial }{\partial t}g_M(t,s)_{\mid t=a}>0$ and $\frac{\partial }{\partial t}g_M(t,s)_{\mid t=b}<0$ ($\frac{\partial }{\partial t}g_M(t,s)_{\mid t=a}<0$ and $\frac{\partial }{\partial t}g_M(t,s)_{\mid t=b}>0$) a.e. on $(a,b)$, if, and only if,   operator $T[M]$  is strongly inverse positive (strongly inverse negative) in $X$.
		\end{theorem}
		
	The following conditions on $g_M(t,s)$ have been introduced in \cite[Section 1.8]{Cab} and they  will be used along the paper.
	\begin{itemize}
		\item[$(P_g$)] Suppose that there is a continuous function $\phi(t)>0$ for all $t\in (a,b)$ and $k_1,\ k_2\in \mathcal{L}^1(I)$, such that $0<k_1(s)<k_2(s)$ for a.e. $s\in I$, satisfying
		\[\phi(t)\,k_1(s)\leq g_M(t,s)\leq \phi(t)\, k_2(s)\,,\quad \text{for a.e. } (t,s)\in I \times I .\]
		\item[$(N_g$)] Suppose that there is a continuous function $\phi(t)>0$ for all $t\in (a,b)$ and $k_1,\ k_2\in \mathcal{L}^1(I)$, such that $k_1(s)<k_2(s)<0$ for a.e. $s\in I$, satisfying
		\[\phi(t)\,k_1(s)\leq g_M(t,s)\leq \phi(t)\, k_2(s)\,,\quad \text{for a.e. } (t,s)\in I \times I .\]
	
	\end{itemize}
	
	Finally, we introduce the following sets, which show the parameter's set where the Green's function is of constant sign,
	\begin{eqnarray}
	\nonumber	P_T&=& \left\lbrace M\in \mathbb{R}\ \mid \quad g_M(t,s)\geq 0\quad \forall (t,s)\in I\times I\right\rbrace, \\
\nonumber		N_T&=& \left\lbrace M\in \mathbb{R}\ \mid \quad g_M(t,s)\leq 0\quad \forall (t,s)\in I\times I\right\rbrace.
	\end{eqnarray}

	Using  Theorem \ref{T::int} we know that these two sets are real intervals. We mention that they are not necessarily bounded or nonempty.

	Next results describe the structure of the two previous real intervals,
	\begin{theorem}\emph{\cite[Theorem 1.8.31]{Cab}} \label{T::6}
		Let $\bar{M}\in \mathbb{R}$ be fixed. If operator $L_n[\bar{M}]$ is invertible in $X_U$ and its related Green's function satisfies condition $(P_g)$, then the following statements hold:
		\begin{itemize}
			
			\item There exists $\lambda_1>0$, the least eigenvalue in absolute value of operator $L_n[\bar{M}]$ in $X_U$. Moreover, there exists a nontrivial constant sign eigenfunction corresponding to the eigenvalue $\lambda_1$.
			\item Green's function related to operator $L_n[\bar{M}]$ is nonnegative on $I\times I$ for all $M\in(\bar{M}-\lambda_1,\bar{M}]$.
			
			\item Green's function related to operator $L_n[\bar{M}]$ cannot be nonnegative on $I\times I$ for all $M<\bar{M}-\lambda_1$.
			\item If there is $M\in \mathbb{R}$ for which Green's function related to operator $L_n[\bar{M}]$ is nonpositive on $I\times I$, then $M<\bar{M}-\lambda_1$.
		\end{itemize}
	\end{theorem}	
	
	\begin{theorem}\emph{\cite[Theorem 1.8.36]{Cab}}\label{T::7}
		Let $\bar{M}\in\mathbb{R}$ be such that $L_n[\bar{M}]$   is invertible on $X_U$ and the related Green's function $g_{\bar{M}}$ satisfies condition $(P_g)$. If the interval $N_T$ is nonempty then $\sup(N_T)=\inf (P_T)$.
	\end{theorem}
	
	Next result, which appears in \cite{CaSaa}, gives us a property of the operator under the disconjugacy hypothesis.
	
	\begin{lemma}\label{L::2}
		Let $\bar{M}\in\mathbb{R}$ be such that $L_n[\bar{M}]\,u(t)=0$ is disconjugate on $I$. Then the following properties are fulfilled:
		\begin{itemize}
			\item If $n-k$ is even, then $L_n[\bar{M}]$ is a inverse positive operator on $X_k$ and its related Green's function, $g_{\bar{M}}(t,s)$, satisfies ($P_g$).
			
			\item If $n-k$ is odd, then $L_n[\bar{M}]$ is a inverse negative operator on $X_k$  and its related Green's function satisfies ($N_g$).
			
		\end{itemize}
	\end{lemma}
	
\section{Strongly inverse positive character of $T[0]$.}\label{Sc::des}

Under the assumption of disconjugation on $I$ of the second order linear differential operator $L_2$ defined in equation \eqref{Ec::2or}, we prove in this section the strongly inverse positive character of operator $T[0]$ in $X$.

Since the differential equation \eqref{Ec::2or} is disconjugate on $I$ we can apply  Theorems \ref{T::4} and \ref{T::22} to write such equation as \[L_2\, u(t) \equiv v_1(t)\,v_2(t)\,\dfrac{d}{dt}\left( \dfrac{1}{v_2(t)}\dfrac{d}{dt}\left( \dfrac{u(t)}{v_1(t)}\right) \right) =0\,,\ t\in I\,,\] where $v_2\in C^1(I)$ and $v_1\in C^2(I)$ are positive functions on $I$. 

Let us see that, in fact, $v_2\in C^3(I)$ and $v_1\in C^4(I)$. Indeed, following the proof of Theorem \ref{T::22},  given in \cite[Chapter 3]{Cop}, we can affirm that \[v_1(t)=y_1(t)\quad \text{and}\quad 
v_2(t)=\dfrac{\left| \begin{array}{cc} y_1(t)&y_2(t)\\y_1'(t)&y_2'(t)\end{array}\right| }{y_1^2(t)}\,,\] where $y_1$ and $y_2$ form a Markov fundamental system of solutions of \eqref{Ec::2or}. Since $p_1\in C^3(I)$ and $p_2\in C^2(I)$, every solution of \eqref{Ec::2or} is of class $C^4(I)$. Then $v_1\in C^4(I)$ and $v_2\in C^3(I)$.

So, trivially we can express $T[0]$ in the following way
\begin{equation}
T[0]\, u(t)\equiv v_1(t)\,v_2(t)\,\dfrac{d}{dt}\left( \dfrac{1}{v_2(t)}\dfrac{d}{dt}\left( \dfrac{u''(t)}{v_1(t)}\right) \right)\,,\quad t\in I\,.\label{e-de-T[0]}
\end{equation}

Realize that if we use the notation of Theorem \ref{T::22}, we have that $v_4(t)=v_2(t)$, $v_3(t)=v_1(t)$, $v_2(t)=1$ and $v_1(t)=1$. In order to avoid more complication with notation, we are going to keep up with the notation of $v_1(t)$ and $v_2(t)$.

In the sequel, we introduce a previous lemma which will allow us to obtain a characterization of the parameter's set where the Green's function, $g_M$, related to operator $T[M]$ in $X$ is of constant sign.

\begin{lemma}\label{L::pg}
	If the  second order linear differential equation \eqref{Ec::2or} is disconjugate on $I$, then operator $T[0]$ is strongly inverse positive on $X$.	

	Moreover, the related Green's function, $g_0(t,s)$, satisfies condition $(P_g)$.
	
\end{lemma}

\begin{proof}
	Firstly, let us see the strongly inverse positive character of operator $T[0]$ on $X$.
	
	Let $u\in X$ be such that $T[0]\,u\gneqq0$ on $I$. Considering now the decomposition expression of operator $T[0]$ given in \eqref{e-de-T[0]}, since $v_1(t)\,v_2(t)>0$ for all $t \in I$, the fact that $T[0]\,u \ngeqq 0$ in $I$ allows us to affirm that $\dfrac{1}{v_2}\dfrac{d}{dt}\left( \dfrac{u''}{v_1}\right)$ is a nondecreasing function in $I$, which can vanish at most once on $(a,b)$.
	
	 Since $v_2(t)>0$ for all $t \in I$, we have that function $\dfrac{d}{dt}\left( \dfrac{u''}{v_1}\right) $ can also vanish at most once on $(a,b)$, being negative at $t=a$ and positive at $t=b$.
	
	So, $\dfrac{u''}{v_1}$ has at most two zeros on $I$. Since ${v_1}>0$, also does $u''$. But $u''(a)=u''(b)=0$, so $u''$ is of constant sign on $(a,b)$. Thus, in order to ensure the maximum number of zeros of $u''$, we have that $\dfrac{d}{dt}\left( \dfrac{u''(t)}{v_1(t)}\right) _{\mid t=a}<0$ and so, we can affirm that $u''<0$ on $(a,b)$.
	
	Hence, $u$ is a concave function on $I$, verifying $u(a)=u(b)=0$, so it must be positive on $(a,b)$. 
	
	If $u'(a)=0$ or $u'(b)=0$, we have that, since $u''<0$ on  $(a,b)$, $u'$ is of constant sign on $I$. Thus $u(a)=u(b)=0$ implies that $u\equiv 0$ on $I$. But, this contradicts the fact that $T[0]\,u\gneqq0$ on $I$. Then $u'(a)>0$ and $u'(b)<0$. Hence,  the strongly inverse positive character of operator $T[0]$ in $X$ is proved.
	
	Moreover, we conclude, by means of Theorem \ref{T::14}, that $g_0(t,s)>0$ a.e. on $I \times I$, and $\frac{\partial }{\partial t}g_0(t,s)_{\mid t=a}>0$ and $\frac{\partial }{\partial t}g_0(t,s)_{\mid t=b}<0$.
	
	Let us see that, in fact,  $g_0(t,s)>0$  on $I \times I$. 
	
	For a fixed $s\in(a,b)$, we denote $u_s(t)\equiv g_0(t,s)$. From the properties of  Green's function, we know that $u_s\geq0$ is a solution of problem
		$$T[0]\,u_s(t)=0, \; t \in I\backslash \{s\}, \quad u_s(a)=u''_s(a)=u_s(b)=u''_s(b)=0.$$
		
	We are going to see that $u_s$ cannot have any double zero on $(a,b)$.
		
		Since $T[0]\,u_s(t)=0$ if $t\neq s$, and $v_1\,v_2>0$ on $I$, $\dfrac{d}{dt}\left( \dfrac{u_s''}{v_1}\right)$ has two constant sign components, which, to allow the maximal oscillation, must be of different sign in each subinterval.

	Since $g_0$ is a $C^2$ function on $I \times I$, we have that $\dfrac{u_s''}{v_1}$  is a continuous function that has at most two zeros on $I$, verifying $u_s''(a)=u_s''(b)=0$. So, it is of constant sign in $(a,b)$. Now, from the positiveness of $v_1$, we can ensure that the same property holds for $u_s''$.
	
	 Hence, $u_s$ has at most two zeros on $I$. Since $u_s(a)=u_s(b)=0$, we deduce that $u_s>0$ on $(a,b)$. So 
	\begin{equation}
	\dfrac{g_0(t,s)}{(t-a)\,(b-t)}>0\,,\quad \forall(t,s)\in (a,b) \times (a,b)\,. \label{e-des-g0}
	\end{equation}
%
%
%
%
%
%
%
%
	 And, for $s\in(a,b)$ the following limits exist
	
		\begin{eqnarray}\nonumber 
		\lim_{t\rightarrow a^+} \dfrac{g_0 (t,s) }{(t-a)\,(b-t)}&=&\dfrac{\frac{\partial }{\partial t}g_0(t,s)_{\mid t=a}}{b-a}=\ell_1(s) ,\\\nonumber\\
		\nonumber \lim_{t\rightarrow b^-} \dfrac{g_0 (t,s) }{(t-a)\,(b-t)}&=&-\dfrac{\frac{\partial }{\partial t}g_0(t,s)_{\mid t=b}}{b-a}=\ell_2(s).
		\end{eqnarray}
	Moreover,  $\ell_1(s)>0$  and $\ell_2(s)>0$ for a.e. $s\in(a,b)$. So, for each $s\in(a,b)$, we construct the continuous extension of $
		\dfrac{g_0(t,s)}{(t-a)(b-t)}$ to $I$, thus
	\[0<k_1(s)=\min_{t\in I} \dfrac{g_0(t,s)}{(t-a)(b-t)}< \max_{t\in I} \dfrac{g_0(t,s)}{(t-a)\,(b-t)}=k_2(s)\,,\quad\text{for a.e. } s\in (a,b)\,. \]
		
		The continuity of functions $k_1$ and $k_2$ follows from the continuity of function $g_0$.
				
		To verify condition $(P_g)$ is enough to take $\phi(t)=(t-a)\,(b-t)$.	
\end{proof}

\section{Construction of the adjoint operator of $T[M]$.}\label{Sc::adj}

To prove the main result of this paper, we need to work with the adjoint operator of $T[M]$, which we denote  as $T^*[M]$.

As it is proved in \cite[Theorem 10, pag. 104]{Cop}, operator $T^*[0]$ has the correspondent decomposition 
\begin{equation}
\label{e-de-T*[0]}
T^*[0]\,v(t)\equiv \dfrac{d^2}{dt^2}\left( \dfrac{1}{v_1(t)}\dfrac{d}{dt}\left( \dfrac{1}{v_2(t)}\dfrac{d}{dt}\left( v_1(t)\,v_2(t)\,v(t)\right) \right) \right) \,,\quad t\in I\,,
\end{equation}
where $v_1$ and $v_2$ are previously introduced in equation \eqref{e-de-T[0]}. 

We note that the regularity of $v_1$ and $v_2$ is required in order to ensure the validity of previous expression.

At first, we are going to construct the space where this adjoint operator is defined. To this end, we use the characterization given in \cite[Section 1.4]{Cab} (see also \cite[page 74]{Cop}).

{\footnotesize \begin{eqnarray} \label{DA}
	X^*&=&\left\lbrace v\in C^4(I)\ \mid \sum_{j=1}^{4}\sum_{i=0}^{j-1} (-1)^{j-1-i} (p_{4-j}\,v)^{(j-1-i)}(b)\,u^{(i)}(b)\right. \\\nonumber
	&&=\left. \sum_{j=1}^{4}\sum_{i=0}^{j-1} (-1)^{j-1-i} (p_{4-j}\,v)^{(j-1-i)}(a)\,u^{(i)}(a) \ \text{ (with $p_0=1$ and $p_3=0$)}\,, \forall u\in X\right\rbrace \,.
	\end{eqnarray}}

Since  $u(a)=u(b)=u''(a)=u''(b)=0$ for every $u\in X$ , we can transform it as follows

{\footnotesize \begin{eqnarray} 
X^*&=&\left\lbrace v\in C^4(I)\ \mid \sum_{j=2}^{4} (-1)^{j-2} (p_{4-j}\,v)^{(j-2)}(b)\,u'(b)+v(b)\,u'''(b)\right. \\\nonumber
	&&=\left. \sum_{j=2}^{4} (-1)^{j-2} (p_{4-j}\,v)^{(j-2)}(a)\,u'(a) +v(a)\,u'''(a)\ \text{ (with $p_0=1$)}\,, \forall u\in X\right\rbrace \,.
	\end{eqnarray}}

If we choose $u\in X$, such that $u'(a)=u'(b)=u'''(b)=0$ and $u'''(a)=1$, we have that $v(a)=0$. So, if $v\in X^*$, then $v(a)=0$.

Analogously, taking $u\in X$, such that $u'(a)=u'(b)=u'''(a)=0$ and $u'''(b)=1$, we have that $v\in X^*$ must verify $v(b)=0$.

Now, if we choose $u\in X$ verifying $u'(a)=u'''(a)=u'''(b)=0$ and $u'(b)=1$, we have that every $u\in X^*$ should satisfy
\[p_2(b)\,v(b)-p_1'(b)\,v(b)-p_1(b)\,v'(b)+v''(b)=0\,,\]
which, since $v(b)=0$ and $p_1'$ and $p_2$ are continuous functions in $I$, is equivalent to
\[v''(b)-p_1(b)\,v'(b)=0\,.\]

 Analogously, with $u\in X$, such that $u'(b)=u'''(a)=u'''(b)=0$ and $u'(a)=1$, we have that every $v\in X^*$ satisfies $v''(a)-p_1(a)\,v'(a)=0$.
 
 So we conclude that the set of definition of the adjoint operator is
 \[X^*=\left\lbrace v\in C^4(I) \mid  v(a)=v''(a)-p_1(a)\,v'(a)=v(b)=v''(b)-p_1(b)\,v'(b)=0\right\rbrace \,.\]
 
 \begin{remark}\label{Rm::3.1}
	Under similar arguments, we can deduce that $X_k^*=X_{n-k}$ (see \cite{CaSaa} for details).
\end{remark}

In the sequel, we will prove that $\dfrac{d}{dt}\left( \dfrac{1}{v_2(t)}\dfrac{d}{dt}\left( v_1(t)\,v_2(t)\,v(t)\right) \right) $ vanish at $t=a$ and $t=b$ for every $v\in X^*$. 
 
 We write the previous expression in the following way
 \[\dfrac{d}{dt}\left( \dfrac{v_2'(t)\,v_1(t)}{v_2(t)}\right) v(t)+\dfrac{v_2'(t)\,v_1(t)}{v_2(t)}\,v'(t)+v_1''(t)v(t)+2v_1'(t)\,v'(t)+v_1(t)\,v''(t)\,,\]
 since $v(a)=v(b)=0$, $v_2\in C^3(I)$ and $v_1\in C^4(I)$ we have
 {\scriptsize \begin{eqnarray}
 \label{Ec::Ad1}\dfrac{d}{dt}\left( \dfrac{1}{v_2(t)}\dfrac{d}{dt}\left( v_1(t)\,v_2(t)\,v(t)\right) \right)_{\mid t=a}&=& \left( \dfrac{v_2'(a)\,v_1(a)}{v_2(a)}+2v_1'(a)\right) \,v'(a)+v_1(a)\,v''(a)\,,\\\label{Ec::Ad2}
 \dfrac{d}{dt}\left( \dfrac{1}{v_2(t)}\dfrac{d}{dt}\left( v_1(t)\,v_2(t)\,v(t)\right) \right)_{\mid t=b}&=&\left(  \dfrac{v_2'(b)\,v_1(b)}{v_2(b)}+2v_1'(b)\right) \,v'(b)+v_1(b)\,v''(b)\,.
 \end{eqnarray}}
 
 Now, to see that  two previous equalities are null, let us write $p_1$ in terms of $v_1$ and $v_2$.
 
 In order to do that, we are going to develop the decomposition of the operator $T[0]$ given in  Section \ref{Sc::des}.
{\scriptsize  \begin{eqnarray}\nonumber v_1(t)\,v_2(t)\dfrac{d}{dt}\left( \dfrac{1}{v_2(t)}\dfrac{d}{dt}\left( \dfrac{u''(t)}{v_1(t)}\right)  \right) &=&v_1(t)\,v_2(t)\dfrac{d}{dt}\left(\dfrac{u'''(t)\,v_1(t)-v_1'(t)\,u''(t)}{v_2(t)\,v_1^2(t)}\right) \\\nonumber
 &=&u^{(4)}(t)-\left( \dfrac{v_2'(t)}{v_2(t)}+2\dfrac{v_1'(t)}{v_1(t)}\right) \,u'''(t)\\\nonumber&&+\left( \dfrac{v_2'(t)\,v_1'(t)}{v_2(t)\,v_1(t)}+2\dfrac{v_1'^2(t)}{v_1^2(t)}-\dfrac{v_1''(t)}{v_1(t)}\right) u''(t)\,.\end{eqnarray}}

So, $p_1(t)=-\left( \dfrac{v_2'(t)}{v_2(t)}+2\dfrac{v_1'(t)}{v_1(t)}\right) $ and we can transform the boundary conditions $v''(a)-p_1(a)\,v'(a)=v''(b)-p_1(b)\,v'(b)=0$ in the following way

\begin{eqnarray}
\label{Ec::Ta*} v''(a)+\left( \dfrac{v_2'(b)}{v_2(b)}+2\,\dfrac{v_1'(b)}{v_1(b)}\right) \, v'(a)=0\,,\\\nonumber\\\label{Ec::Tb*}
v''(b)+\left( \dfrac{v_2'(b)}{v_2(b)}+2\,\dfrac{v_1'(b)}{v_1(b)}\right) \,v'(b)=0\,.
\end{eqnarray}

So, multiplying  equations \eqref{Ec::Ad1}--\eqref{Ec::Ad2} by $\dfrac{1}{v_1(a)}$ and $\dfrac{1}{v_1(b)}$ respectively, we obtain exactly the boundary conditions \eqref{Ec::Ta*} and \eqref{Ec::Tb*}. So, both of them are null.

\section{Study of the eigenvalues of operator $T[0]$ in different spaces of definition.}\label{Sc::aut}
 In this section we will prove that the first positive eigenvalues of $T[0]$ in the spaces
 \begin{equation}\label{Esp::U} U_{[a,b]}=\left\lbrace u\in C^4 (I)\quad\mid\quad u(a)=u'(a)=u(b)=u''(b)=0\right\rbrace \,,\end{equation} 
 and  	
 \begin{equation}\label{Esp::V}V_{[a,b]}=\left\lbrace u\in C^4 (I)\quad\mid\quad u(a)=u''(a)=u(b)=u'(b)=0\right\rbrace \,,\end{equation}
 have an associated eigenfunction of constant sign.

 In order to do that, we introduce some preliminary Lemmas.
 
 \begin{lemma}\label{L::10}
 	If the second order linear differential equation \eqref{Ec::2or} is disconjugate on an interval $[c,d]$, then there exist:
 	\begin{itemize}
 		\item ${\lambda_3'}_{[c,d]}>0$ the least positive eigenvalue of $T[0]$ in $U_{[c,d]}$. Moreover there exist a nontrivial constant sign eigenfunction corresponding to the eigenvalue ${\lambda_3'}_{[c,d]}$.
 	
 		\item  ${\lambda_3''}_{[c,d]}>0$ the least positive eigenvalue of $T[0]$ in $V_{[c,d]}$.  Moreover there exist a nontrivial constant sign eigenfunction corresponding to the eigenvalue ${\lambda_3''}_{[c,d]}$.
 	\end{itemize}
 	
 	Furthermore,
 	\begin{itemize}
 		\item For $M\in[-{\lambda_3'}_{[c,d]},0]$, every nontrivial solution of $T[M]\,u(t)=0$, $t\in[c,d]$, verifying one of the following  boundary conditions
 		\begin{eqnarray}\label{Ec::cfuc}u(c)=u(d)=u''(d)&=&0\,,\\\label{Ec::cfu'c}u'(c)=u(d)=u''(d)&=&0\,,\end{eqnarray}
 		does not have any zero on $(c,d)$.
 			\item For $M\in[-{\lambda_3''}_{[c,d]},0]$, every nontrivial solution of $T[M]\,u(t)=0$, $t\in[c,d]$, verifying one of the following boundary conditions
 			\begin{eqnarray}\label{Ec::cfud}u(c)=u''(c)=u(d)&=&0\,,\\\label{Ec::cfu'd}u(c)=u''(c)=u'(d)&=&0\,,\end{eqnarray}
 		
 			does not have any zero on $(c,d)$.
 	\end{itemize}
 \end{lemma}
 \begin{proof}
Since the linear differential equations \eqref{Ec::2or} and $u''(t)=0$ are disconjugate on $[c,d]$, we can apply Theorem \ref{T::comp} to affirm that $T[0]\,u(t)=0$ is a disconjugate equation in  $[c,d]$ too. 

Then, using Lemma \ref{L::2} and Theorem \ref{T::6} we can affirm that there exists a positive eigenvalue of $T[0]$ in ${X_2}_{[c,d]}$ (with obvious notation) which an associated eigenfunction of constant sing, i.e., verifying the boundary conditions $u(c)=u'(c)=u(d)=u'(d)=0$.
 
 So, we study for $M\leq 0$ the behavior of solutions of problem \[T[M]\,z_M(t)=0\,,\ t\in[c,d]\,,\quad z_M(c)=z_M'(c)=z_M(d)=0\,.\] 
 
 First, let us see what happens for $M=0$.
 
 Let us use the decomposition given in \eqref{e-de-T[0]} to see how a nontrivial function $z_0  \in C^4([c,d])$, satisfying
 \[T[0]z_0(t)=0, \; t \in [c,d], \quad z_0(c)=z_0'(c)=z_0(d)=0,\]
behaves.
 
 Since $v_2\,v_1>0$ on $[c,d]$, we have that $\frac{1}{v_2}\frac{d}{dt}\left( \frac{z_0''}{v_1}\right)$ is a constant function on $[c,d]$.
 
 Now, taking into account that $v_2>0$,  we deduce that $\frac{d}{dt}\left( \frac{z_0''}{v_1}\right)$ is a  constant sign function  and $\dfrac{z_0''}{v_1}$ is a monotone function on $[c,d]$, which can have at most a zero. Since $v_1>0$, this last property of maximum number of zeros also holds for $z_0''$.
 
 Then, $z_0$ can have at most three zeros counting according to their multiplicity, which in fact it has, since $z_0(c)=z_0'(c)=z_0(d)=0$. Then, we can affirm that necessarily $z_0''(c)z_0''(d)<0$ to ensure the maximum number of zeros.
 
 Let use  assume that $z_0(t)\geq 0$, if $z_0(t)\leq 0$ the arguments are analogous.
 
 Since $z_0(c)=z_0'(c)=0$, we can affirm that $z_0''(c)>0$, so $z_0''(d)<0$.
 
 As consequence, if we move $M\leq 0$ starting at zero, we will arrive to  $M=-\lambda$, where $\lambda$ is the least positive eigenvalue of $T[0]$ in ${X_2}_{[c,d]}$. So, \[T[-\lambda]z_{-\lambda}(t)=0\,,\ t\in I\,,\quad z_{-\lambda}(c)=z_{-\lambda}'(c)=z_{-\lambda}(d)=z_{-\lambda}'(d)=0\,,\] and  $z_{-\lambda}(t)>0$ on $(c,d)$. Obviously, we have that $z_{-\lambda}''(d)\geq 0$. Then, it must exist ${\lambda_3'}_{[c,d]}\in (0,\lambda]$, such that  \[T[-\lambda_3']z_{-\lambda_3'}(t)=0\,,\ t\in I\,,\quad z_{-{\lambda_3'}_{[c,d]}}(c)=z_{-{\lambda_3'}_{[c,d]}}'(c)=z_{-{\lambda_3'}_{[c,d]}}(d)=z_{-{\lambda_3'}_{[c,d]}}''(d)=0\,,\] i.e., there is a positive eigenvalue, ${\lambda_3'}_{[c,d]}$, of $T[0]$ in $U_{[c,d]}$.
 
 Analogously, we can see that there is a positive eigenvalue of $T[0]$ in $V_{[c,d]}$, ${\lambda_3''}_{[c,d]}\in(0,\lambda]$.
 
 The fact that the associated eigenfunctions are of constant sign follows from the second part of the proof.

 \vspace{0.8cm}
 
 Now, let us prove the second part of the result.
 
  First, let us use the decomposition given in \eqref{e-de-T[0]} to see that every nontrivial function $u_0  \in C^4([c,d])$, satisfying
  $$T[0]u_0(t)=0, \; t \in [c,d]\,,$$
  coupled with boundary conditions either \eqref{Ec::cfuc} or \eqref{Ec::cfu'c} is of constant sign in $[c,d]$.
  
  Arguing as above, we deduce that $\dfrac{u_0''}{v_1}$ is a monotone function on $[c,d]$, which verifies $\dfrac{u_0''(d)}{v_1(d)}=0$. Thus,  it is of constant sign and, since $v_1(t)>0$ for all $t\in [c,d]$, also does $u_0''$.
  
 As consequence, we have proved that $u_0$ is a strictly concave or convex function which can have at most two zeros on $[c,d]$. With boundary conditions \eqref{Ec::cfuc}, using that $u_0(c)=u_0(d)=0$ we conclude that $u_0$ is of constant sign on $[c,d]$.
  
 We point out that if $u_0'(c)=0$ or $u_0'(d)=0$, since $u_0''$ is of constant sign on $(c,d)$, we have that $u_0'$ is a function of constant sign on $[c,d]$, then $u_0(c)=u_0(d)=0$ implies that $u_0\equiv 0$ on $[c,d]$. In other words, $0$ is not an eigenvalue of operator $T[0]$ coupled either with boundary conditions $u(c)=u'(c)=u(d)=u''(d)=0$ or  $u(c)=u(d)=u'(d)=u''(d)=0$.
 
 Now, working with boundary conditions \eqref{Ec::cfu'c}, we know that $u'_0(c)=0$, then with the previous argument $u_0'$ is of constant sign. Hence, $u_0$ is a monotone function with at most a zero, which satisfies $u_0(d)=0$, then $u_0$ is of constant sign on $[c,d)$. 
 
 \vspace{0.8cm}
 
  Let us see that every solution \[T[M]\,u_M(t)=0\,,\quad t\in[c,d]\,,\] verifying the boundary conditions \eqref{Ec::cfuc} will be of constant sign until one of the following assertions holds: $u_M'(c)=0$ or $u_M'(d)=0$.
  
  From the properties deduced for $M=0$, we know that $u_M >0$  or $u_M<0$ on $(c,d)$ for all $M$ in a suitable interval containing $0$ as interior point. We consider the case where $u_M>0$, the other is analogous.
  
  For all $M$ on that interval we have that\[T[0]\,u_M(t)=-M\,u_M(t)\,,\quad t\in[c,d]\,,\] and, hence $T[0]\,u_M$ is of constant sign on $(c,d)$. Since $v_1\,v_2>0$ on $[c,d]$, $\dfrac{1}{v_2}\dfrac{d}{dt}\left( \dfrac{u_M''}{v_1}\right)$ is a monotone function, that vanish at most once on $[c,d]$. Thus, $\dfrac{d}{dt}\left( \dfrac{u_M''}{v_1}\right)$ can have at most one zero on $[c,d]$ too. Then $\dfrac{u_M''}{v_1}$ has at most two zeros, being zero at $t=d$, so it can have only a change of sign on $[c,d]$, and  the same happens to $u_M''$, then $u_M'$ can have two zeros at most, and $u_M$ has three or less zeros in $[c,d]$. Using that $u_M(c)=u_M(d)=0$, we deduce that it cannot have a double zero in $(c,d)$ while it is of constant sign.
  
  The proof that every solution \[T[M]\,u_M(t)=0\,,\quad t\in[c,d]\,,\] verifying the boundary conditions \eqref{Ec::cfu'c} will be of constant sign until one of the following assertions holds: $u_M(c)=0$ or $u_M'(d)=0$ is analogous.
  
  Analogously, it can be seen that a solution of
  \begin{equation}
  T[M]\,v_M(t)=0, \,  t \in [c,d]\,,
  \end{equation} 
  verifying either the boundary conditions \eqref{Ec::cfud} or \eqref{Ec::cfu'c} cannot have a double zero on $(c,d)$ while it is of constant sign, and, moreover the change's sign becomes on the boundary, $v_M'(c)=0$ or $v_M'(d)=0$ for \eqref{Ec::cfud} and $v_M'(c)=0$ or $v_M(d)=0$ for \eqref{Ec::cfu'd}
  
  It is important to point out that in both situations we do not impose any sign conditions to the real parameter $M$.

  Since, due to Theorem \ref{T::10}, for $M<0$ it cannot exist a nontrivial solution of \[T[M]\,u(t)=0\,,\quad t\in[c,d]\,,\] coupled either with  boundary conditions $u(c)=u'(c)=u''(c)=u(d)=0$ or $u(c)=u(d)=u'(d)=u''(d)=0$ ; we can affirm that  the eigenfunctions associated to the least positive eigenvalues in $U_{[c,d]}$ and $V_{[c,d]}$ are of constant sign on $(c,d)$ and that for every $M\in[-{\lambda_3'}_{[c,d]},0]$ every solution of $T[M]\,u_M(t)=0$ verifying  either the boundary conditions \eqref{Ec::cfuc} or \eqref{Ec::cfu'c} is of constant sign. Analogously, for every $M\in[-{\lambda_3''}_{[c,d]},0]$ every solution of $T[M]\,v_M(t)=0$ verifying either the boundary conditions \eqref{Ec::cfud} or \eqref{Ec::cfu'd} is of constant sign. 
 \end{proof}
   

 \begin{remark}\label{Rem::1}
 	Arguing as in Section \ref{Sc::adj} we can obtain the adjoint spaces of definition of $U_{[a,b]}$ and $V_{[a,b]}$. They are given by
 	\begin{eqnarray}
 	\nonumber U^*_{[a,b]}&=&\left\lbrace u\in C^4(I)\quad\mid u(a)=u'(a)=u(b)=u''(b)-p_1(b)\,u'(b)=0\right\rbrace\,, \\\nonumber
 	V^*_{[a,b]}&=&\left\lbrace u\in C^4(I)\quad\mid u(a)=u''(a)-p_1(a)\,u'(a)=u(b)=u'(b)=0\right\rbrace\,.
 	\end{eqnarray}
 	
 	\end{remark}

 Using similar arguments we obtain the following Lemma.
 
  \begin{lemma}\label{L::11}
  	If the second order linear differential equation \eqref{Ec::2or} is disconjugate on $[c,d]$, there exist:
  	\begin{itemize}
  		\item ${\lambda_3'}_{[c,d]}>0$ the least positive eigenvalue of $T^*[0]$ in $U_{[c,d]}^*$.  Moreover there exist a nontrivial constant sign eigenfunction corresponding to the eigenvalue ${\lambda_3'}_{[c,d]}$.
  		\item  ${\lambda_3''}_{[c,d]}>0$ the least positive eigenvalue of $T^*[0]$ in $V_{[c,d]}^*$.  Moreover there exist a nontrivial constant sign eigenfunction corresponding to the eigenvalue ${\lambda_3''}_{[c,d]}$.
  	
  	\end{itemize}
  	
  	Furthermore,
  	\begin{itemize}
  		\item For $M\in[-{\lambda_3'}_{[c,d]},0]$, every nontrivial solution of $T^*[M]\,v(t)=0$, verifying one of the following boundary conditions
  		
  		\[v(c)=v(d)=v''(d)-p_1(d)\,v'(d)=0\,,\]
  		does not have any zero on $(c,d)$.
  		\item For $M\in[-{\lambda_3''}_{[c,d]},0]$, every nontrivial solution of $T^*[M]\,v(t)=0$, verifying the boundary conditions
  		\[v(c)=v''(c)-p_1(c)\,v'(c)=v(d)=0\,,\]
  		does not have any zero on $(c,d)$.
  	\end{itemize}
  \end{lemma}
 \begin{proof}
 	Using \eqref{Ec::gg} we can affirm that the least positive eigenvalue of $T^*[0]$ in $U^*_{[c,d]}$ exists and it coincides with ${\lambda_3'}_{[c,d]}$ and the least positive eigenvalue of $T^*[0]$ in $V^*_{[c,d]}$ exists and it is ${\lambda_3''}_{[c,d]}$, where ${\lambda_3'}_{[c,d]}$ and ${\lambda_3''}_{[c,d]}$ are given in Lemma \ref{L::10}
 		
 	Now,  from the calculation made at Section \ref{Sc::adj}, we can affirm that the boundary conditions 
	
$u(c)=u''(c)-p_1(c)\,u'(c)=0$, imply  { $\dfrac{d}{dt}\left( \dfrac{1}{v_2(t)}\dfrac{d}{dt}\left( v_2(t)v_1(t)v(t)\right) \right)_{\mid t=c}=0 $} 
	and 
	
	$u(d)=u''(d)-p_1(d)\,u'(d)=0$, imply  $\dfrac{d}{dt}\left( \dfrac{1}{v_2(t)}\dfrac{d}{dt}\left( v_2(t)\,v_1(t)\,v(t)\right) \right)_{\mid t=d}=0 $.
	
	The end of the proof follows the same arguments as Lemma \ref{L::10}, using, in this situation, the decomposition of the adjoint operator given in \eqref{e-de-T*[0]}.
 \end{proof}
 
As consequence of these results we obtain a property of the existence of eigenvalues in some spaces.

\begin{lemma}\label{L::12}
If the second order linear differential equation \eqref{Ec::2or} is disconjugate on $[c,d]$, the following assertions are verified:
	
	\begin{itemize}
	
\item For all $M\in(-{\lambda_3'}_{[c,d]},0]$ and $e\in[c,d)$, the following problem has no nontrivial solution $u \in C^4([e,d])$:
	\begin{equation}
	T[M]\,u(t)=0, \,  t \in [e,d], \quad u(e)=u'(e)=u(d)=u''(d)=0.
	\end{equation}

\item For all $M\in(-{\lambda_3''}_{[c,d]},0]$ and $e\in(c,d]$, the following problem has no nontrivial solution $u \in C^4([e,d])$:
	$$T[M]\,u(t)=0, \;  t \in [c,e], \quad u(c)=u''(c)=u(e)=u'(e)=0,$$

	\end{itemize}
\end{lemma}
\begin{proof}
	We are going to prove the first assertion, the proof of the second one is analogous.

To this end, we get the fundamental system of solutions $y_1[M](t)$, $y_2[M](t)$, $y_3[M](t)$, $y_4[M](t) $, where $y_i^{(i-1)}(d)=1$ and $y_{i}^{(k)}(d)=0$ for $k\in\{0,1,2,3\}$, $k\neq i-1$, and we construct the Wronskian
\[W_M(t)=\left| \begin{array}{cc}
y_2[M](t)&y_4[M](t)\\&\\
y_2'[M](t)&y_4'[M](t)\end{array}\right| \,,\]
which is a continuous function with respect to $M$.

%

It is obvious that the general solution of problem 
$$T[M]\,y(t)=0\,,\quad t \in [c,d], \quad y(d)=y''(d)=0,$$
 is given by the expression
\[y(t)=\alpha_1\,y_2[M](t)+\alpha_2\,y_4[M](t),\]
with $\alpha_1, \; \alpha_2 \in \R$.

Every nontrivial solution given by the previous expression, verifies $y(e)=y'(e)=0$ if, and only if, $W_M(e)=0$.

Moreover, taking into account the proof of Lemma \ref{L::10}, we can affirm that $W_0(e)\neq 0$ for every $e\in[c,d)$.

Also, since we have seen that the first positive eigenvalue of $T[0]$ in $U_{[c,d]}$ is ${\lambda_3'}_{[c,d]}$, we can affirm that $W_M(c)\neq 0$ for every $M\in(-{\lambda_3'}_{[c,d]},0]$.

Then if there exist  $M\in(-{\lambda_3'}_{[c,d]},0]$  and  $e \in [c,d]$ for which $W_M(e)=0$. Then the following problem has a nontrivial solution 
$$T[M]\,u(t)=0, \;  t \in [e,d], \quad u(e)=u'(e)=u(d)=u''(d)=0.$$

Since $W_M(t)$ is a continuous function of $M$, it must exist $\bar{M}\in (-{\lambda_3'}_{[c,d]},0]$ and  $\bar{e}\in [c,d]$ such that $W_{\bar{M}}(\bar{e})=0$ and $W_{\bar{M}}'(\bar{e})=0$, i.e.
	\[\left| \begin{array}{cc}
	y_2[\bar{M}](\bar{e})&y_4[\bar{M}](\bar{e})\\&\\&\\
		y_2'[\bar{M}](\bar{e})&y_4'[\bar{M}](\bar{e})\end{array}\right|=0 \quad \text{and}\quad \left| \begin{array}{cc}
		y_2[\bar{M}](\bar{e})&y_4[\bar{M}](\bar{e})\\&\\&\\
		y_2''[\bar{M}](\bar{e})&y_4''[\bar{M}](\bar{e})\end{array}\right|=0\,.\]
		
		Then, if we consider  function
		\[y(t)=y_4[\bar{M}](\bar{e})\,y_2[\bar{M}]\,(t)-y_2[\bar{M}](\bar{e})\,y_4[\bar{M}]\,(t)\,,\]
		it satisfies $T[M]\,y(t)=0$ on $[c,d]$ coupled with the boundary conditions $y(\bar{e})=y'(\bar{e})=y''(\bar{e})=y(d)=y''(d)=0$. In particular, it satisfies the boundary conditions $3-1$ in $[e,d]$ and, as consequence, there exists a positive eigenvalue of $T[0]$ in ${X_3}_{[c,d]}$.
		
		Now, since the linear differential equations \eqref{Ec::2or} and $u''(t)=0$ are disconjugate on $[e,d]$, we can apply  Theorem \ref{T::comp} to affirm that $T[0]\,u(t)=0$ is a disconjugate equation in $[e,d]$ too. So, since $n-k=1$ is an odd number we attain a contradiction with Theorem \ref{T::10}.	
\end{proof}

As a direct corollary from this property we obtain the following result.
		\begin{corollary}\label{Cor::2}
			If the second order linear differential equation \eqref{Ec::2or} is disconjugate on $[c,d]$, then:
		\begin{itemize}
			\item The least positive eigenvalue of $T[0]$ in $U_{[e,d]}$, ${\lambda_3'}_{[e,d]}$, increases with respect to $e\in[c,d)$.
			\item The least positive eigenvalue of $T[0]$ in $V_{[c,e]}$, ${\lambda_3''}_{[c,e]}$, decreases with respect to $e\in(c,d]$.
		\end{itemize}
		\end{corollary}
		
\section{Main result}
This section is devoted to prove the main result of this paper. The result is the following.
\begin{theorem}\label{T::1}
	If the second order linear differential equation \eqref{Ec::2or} is disconjugate on $I$, then $T[M]$ is strongly inverse positive in $X$ if, and only if, $M\in(-\lambda_1,-\lambda_2]$, where:
	\begin{itemize}
		\item $\lambda_1>0$ is the least positive eigenvalue of $T[0]$ in $X$.
		\item $\lambda_2<0$ is the maximum between:
		\begin{itemize}
			\item $\lambda_2'<0$, the biggest negative eigenvalue of $T[0]$ in $X_1$.
			\item $\lambda_2''<0$, the biggest negative eigenvalue of $T[0]$ in $X_3$. 
		\end{itemize}
	\end{itemize}
	
Moreover, $T[M]$ is strongly inverse negative in $X$ if, and only if, $M\in[-\lambda_3,-\lambda_1)$, where:
	\begin{itemize}
		\item $\lambda_3>0$ is the minimum between:
		\begin{itemize}
			\item $\lambda_3'\equiv{\lambda_3'}_{[a,b]}>0$, the least positive eigenvalue of $T[0]$ in $U_{[a,b]}$.
		
			\item $\lambda_3''\equiv{\lambda_3''}_{[a,b]}>0$, the least positive eigenvalue of $T[0]$ in $V_{[a,b]}$.
			
		\end{itemize}
	\end{itemize}
\end{theorem}
\begin{proof}
 The inverse positive character for $M\in(-\lambda_1,0]$ follows from Lemma \ref{L::pg} and Theorem \ref{T::6}. Since $T[0]$ is a strongly inverse positive operator, from Theorems \ref{T::14} and \ref{T::d1} we have that the strongly inverse positive character also holds for $M\in(-\lambda_1,0]$.
 
 Once it is proved that $N_T$ is not empty, from Theorem \ref{T::7} we deduce that $\sup(N_T)=-\lambda_1$.

 The proof of the other extremes of the intervals $P_T$ and $N_T$ follows several steps.
 
  At first, we focus on the inverse positive parameter set, $P_T$. We study what happens on the boundary of $I\times I$ and, after that, we prove that the sign change must begin necessarily on the boundary. 
  
  Once we have described the interval $P_T$, we characterize $N_T$ by proving at first that there must exist an interval where the Green's function is nonpositive on a neighborhood of the boundary. And, to finish, we will see that while it is non positive on the boundary it must be nonpositive on the interior too.
  
\begin{itemize}
	\item[Step 1.] Behavior of the Green's function at $s=a$.
\end{itemize}

From equation \eqref{Ec::gg}, we know that $g_M(t,s)=g_M^*(s,t)$, where $g_M^*(t,s)$ is the Green's function related to the adjoint operator $T^*[M]$.

So, taking into account the boundary conditions verified for the adjoint operator, given in Section \ref{Sc::adj}, we know that $g_M(t,a)=g_M^*(a,t)=0$.

So, we define the following function
\begin{equation}\label{Ec::wm}w_M(t):=\dfrac{\partial }{\partial s}g_M^2(t,s)_{\mid s=a}\,,t\in I\end{equation}
where
\[g_M(t,s)=\left\lbrace \begin{array}{cc}
g_M^1(t,s)&\text{ if }  a\leq t<s\leq b\,,\\&\\
g_M^2(t,s)& \text{ if } a\leq s\leq t\leq b\,.\end{array}\right. \]

It is clear that, while $w_M> 0$  on $(a,b)$, then $g_M(t,s)$ satisfies the property
{\small \begin{equation}\label{Ec::pa}
	\forall\, t\in(a,b)\quad \exists \,\eta(t)>0 \text{ such that } g_M(t,s)=g_M^2(t,s)>0\ \forall s\in (a,a+\eta(t))\,.
	\end{equation}}
 
 Reciprocally, if $w_M$ oscillates on $I$, then $g_M(t,s)$ oscillates on a neighborhood of $s=a$.

 In order to deduce which are the boundary conditions that $w_M$ satisfies, we will use the expression of the Green's matrix related to the vectorial problem \eqref{Ec::vec} given in \eqref{Ec:MG1}.

Using the boundary conditions \eqref{Ec::Cf}, taking into account the expression of $g_3(t,s)$ given in \eqref{Ec::g3}, we have that $g_M^1(t,s)$ satisfies the following equations for $s\in(a,b)$
\[ g_M^1(a,s)=
-\dfrac{\partial}{\partial s}g_M^1(t,s)_{\mid t=a}+p_1(s)g_M^1(a,s)=0\,,\]
 and, because of the continuity for the Green matrix on the elements which do not belong to the diagonal, previous equations are verified for $s=a$ and we have

\[g_M^2(a,a)=-\dfrac{\partial}{\partial s}g_M^2(t,s)_{\mid (t,s)=(a,a)}+p_1(a)g_M^2(a,a)=0\,.\]

In particular, we  conclude that $w_M(a)=0$.
	
	Now, studying the third row of the correspondent matrix \eqref{Ec:MG1} we arrive to
	
	\[\dfrac{\partial}{\partial t^2}g_M^1(t,s)_{\mid t=a}=
	-\dfrac{\partial^3}{\partial s\partial t^2}g_M^1(t,s)_{\mid t=a}+p_1(s)\dfrac{\partial}{\partial t^2}g_M^1(t,s)_{\mid t=a}=0\,,\quad s\in(a,b)\]
	we can again consider such equalities for $s=a$, but now, since we achieve a diagonal element, by applying Remark \ref{Rm::2.5}, in terms of $g_M^2(t,s)$ we have the following system
	
		\begin{eqnarray}
		\nonumber \dfrac{\partial}{\partial t^2}g_M^2(t,s)_{\mid (t,s)=(a,a)}&=&0\,,\\\nonumber
		-\dfrac{\partial^3}{\partial s\partial t^2}g_M^2(t,s)_{\mid (t,s)=(a,a)}+p_1(s)\dfrac{\partial}{\partial t^2}g_M^2(t,s)_{\mid (t,s)=(a,a)}&=&1\,,
		\end{eqnarray}
		so, we conclude that $w_M''(a)=-1$.
		
	Analogously, we study the boundary conditions for $w_M$ at $t=b$. Now, since in this case we do not have any jump at $s=a$, the obtained system is already expressed by means of $g_M^2(t,s)$. We directly show the equations for $s=a$.
	
	\[ g_M^2(b,a)=-\dfrac{\partial}{\partial s}g_M^2(t,s)_{\mid (t,s)=(b,a)}+p_1(a)g_M^2(b,a)=0\,.\]
	
		So, we obtain $w_M(b)=0$ and, similarly, $w_M''(b)=0$.

	Thus, $w_M$ satisfies the boundary conditions
	
	\begin{equation}\label{Ec::cfW}
	w_M(a)=w_M(b)=w_M''(b)=0\,,\quad w_M''(a)=-1\,.\end{equation}
	
	 As consequence, for the particular case of $M=0$, since $g_0(t,s)\geq 0$,  using Lemma \ref{L::10} we have that $w_0(t)> 0$ for all $t\in (a,b)$.

	Now, by analogous arguments to the ones used in the proof of Lemma \ref{L::10} we can affirm that the sign change must begin on the boundary of $I$. In  this case, we are interested on the behavior for $M\geq 0$, because for $M\leq 0$ the result is already known using Theorem \ref{T::6}.

	So, necessarily the function $w_M>0$ on $(a,b)$ until $w_M'(a)=0$ or $w_M'(b)=0$.
	
	If $w_M'(a)=0$, since $w_M(a)=0$ and $w_M''(a)=-1$, we have that necessarily $w_M(t)<0$ on a neighborhood of $t=a$, so it must have changed sign before. Hence, $w_M>0$ on $(a,b)$ until it is verified $w_M(a)=w_M(b)=w_M'(b)=w_M''(b)=0$, i.e., for $M\in[0,-\lambda_2']$. Hence, we can affirm that for $M$ in such interval $g_M(t,s)$ satisfies \eqref{Ec::pa}.
	
		Let us see that for $M>-\lambda_2'$, $g_M(t,s)$ oscillates.
		
		 Suppose that $g_{\bar{M}}(t,s)\geq0$ for some ${\bar{M}}>-\lambda_2'$ and all $(t,s)\in I\times I$. From Theorem \ref{T::d1} we have  $g_{-\lambda_2'}\geq g_{\bar{M}}$. Since, on the other hand, $g_{\bar{M}}(t,a)=0$, we know that $w_{\bar{M}}(t)\leq w_{-\lambda_2'}(t)$ on $I$, then necessarily $w_{\bar{M}}'(b)\geq w_{-\lambda_2'}'(b)=0$.
		 
		 Notice that if $w_{\bar{M}}'(b)>0$, then $w_{\bar{M}}(t)<0$ on a neighborhood of $t=b$, so $g_M(t,s)$ change sign on a neighborhood of $(b,a)$. As consequence, $w_{\bar{M}}'(b)=0$.
		 
		  Using Theorem \ref{T::int}, we know that for every $M\in[-\lambda_2',\bar{M}]$, $g_M(t,s)\geq0$ for all $(t,s)\in I\times I$, so $w_M'(b)=0$. This implies that every $M\in [-\lambda_2',\bar{M}]$ is an eigenvalue of $T[0]$ in $X_1$, which contradicts the discrete character of the eigenvalues' set.
		 Hence, we can conclude that for $M>-\lambda_2'$, $g_M(t,s)$ oscillates on a neighborhood of $s=a$.
	
	\begin{itemize}
		\item[Step 2.] Behavior on a neighborhood of $s=b$.
	\end{itemize}
		
	Now, since $g_M(t,s)=g_M^*(s,t)$ and taking into account the boundary conditions given for the adjoint operator in Section \ref{Sc::adj}, we have that $g_M(t,b)=g_M^*(b,t)=0$. So, we define
	\begin{equation}\label{Ec::ym}y_M(t):=\dfrac{\partial }{\partial s}g_M(t,s)_{\mid s=b}\,,\end{equation}
	and with similar arguments, using the expression of the Green's matrix given by \eqref{Ec:MG1} and the expressions of  $g_3$ given in \eqref{Ec::g3} we have that $y_M$ satisfies the following boundary conditions
	\begin{equation}\label{Ec::cfym}y_M(a)=y_M(b)=y_M''(a)=0\,,\quad y_M''(b)=1\,.\end{equation}
	
	In this case, for $M=0$, from Lemma \ref{L::10}, we have that  $y_0< 0$ on $(a,b)$.
	
	With the same arguments as in  Step 1 we arrive at the conclusion that  Green's function must satisfy the following property
	 {\small \begin{equation}\label{Ec::pb}
	 	\forall\, t\in(a,b)\quad \exists\, \eta(t)>0 \text{ such that } g_M(t,s)=g_M^1(t,s)>0\ \forall\, s\in (b-\eta(t),b)\,
	 	\end{equation}}
	 while $M\in[0,-\lambda_2'']$, where $\lambda_2''<0$ is the biggest negative eigenvalue of $T[0]$ in $X_3$. And, we can also affirm that $g_M(t,s)$ oscillates for every $M>-\lambda_2''$.
	
	So, with this two steps we have already seen that the upper bound given by $-\lambda_2$ cannot be improved. 
	
	\begin{itemize}
		\item[Step 3.] Behavior of the Green's function on a neighborhood of $t=a$ and $t=b$.
	\end{itemize}
	
	In order to study the behavior of the Green's function on a neighborhood of $t=a$ and $t=b$, we are going to consider the adjoint operator and study, with analogous arguments to Steps 1 and 2, the behavior of $g_M^*(t,s)$ in a neighborhood of $s=a$ and $s=b$. Then, using the equality $g_M(t,s)=g_M^*(s,t)$, we can know how $g_M(t,s)$ behaves on a neighborhood of $t=a$ and $t=b$.
	
As we have seen in  Section \ref{Sc::adj}, the boundary conditions satisfied by the adjoint operator are the following
	\[v(a)=v(b)=v''(a)-p_1(a)\,v'(a)=v''(b)-p_1(b)\,v'(b)=0\,.\]
	
	Since $g_M^*(t,a)=g_M(a,t)=0$, we consider the function
	\begin{equation}\label{Ec::wm*}w_M^*(t):=\dfrac{\partial }{\partial s}g_M^*(t,s)_{\mid s=a}\,,\end{equation}
	 studying the corresponding Green's matrix of the adjoint vectorial operator, we can verify that it satisfies the following boundary conditions.
	
	\[w_M^*(a)=w_M^*(b)={w_M^*}''(b)-p_1(b)\,{w_M^*}'(b)=0\,,\quad {w_M^*}''(a)-p_1(a)\,{w_M^*}'(a)=-1\,.\]
	
	Moreover, for $M=0$, we know that $w_0^*> 0$ on $(a,b)$.
	
	With the arguments done in Section \ref{Sc::adj} it is known that $ {w_M^*}''(b)-p_1(b)\,{w_M^*}'(b)=0$ implies that $
	\dfrac{d}{dt}\left( \dfrac{1}{v_2(t)}\dfrac{d}{dt}\left( v_2(t)\,v_1(t)\,w_M^*(t)\right) \right)_{\mid t=b}=0$.
	
	Using this fact we can use the decomposition of the adjoint operator, given in \eqref{e-de-T*[0]} and the arguments of the proof of Lemmas \ref{L::10} and \ref{L::11} to conclude that $w_M^*$ cannot have any double zero while $w_M^*\geq0$ on $I$.
	
	Also, we can see that while ${w_M^*}'(b)\neq 0$, $w_M^*$ must be positive on $(a,b)$.
	
	So, we can affirm that the sign change come when $w_M^*$ satisfies the following boundary conditions:
	\[w_M^*(a)=w_M^*(b)={w_M^*}'(b)={w_M^*}''(b)-p_1(b)\,{w_M^*}'(b)=0\,,\]
	which are equivalent, since $p_1$ is a continuous function, to
	\[w_M^*(a)=w_M^*(b)={w_M^*}'(b)={w_M^*}''(b)\,.\]
	
	Hence, $w_M^*(t)\geq 0$ for $M\in[0,-\lambda_2''']$, where $\lambda_2'''<0$ is the biggest negative eigenvalue of $T^*[0]$ in $X_1$. But, since the eigenvalues of $T^*[0]$ in $X_1$ are the same as the eigenvalues of $T[0]$ in $X_{3}$  (see Remark \ref{Rm::3.1}), we have that $\lambda_2'''=\lambda_2''$.
	
	So, $g_M^*(t,s)$ satisfies the property \eqref{Ec::pa} for all $M\in[0,-\lambda_2'']$ and it oscillates for $M>-\lambda_2''$. So, we can conclude that if $M\in[0,-\lambda_2'']$, then $g_M(t,s)$ satisfies the condition

	{\small \begin{equation}\label{Ec::paa}
		\forall \,s\in(a,b)\quad \exists\, \eta(s)>0 \text{ such that } g_M(t,s)=g_M^2(t,s)>0 \ \forall\, t\in (a,a+\eta(s))
		\end{equation}}

	 Now, proceeding analogously with the function 
	 \begin{equation}\label{Ec::ym*}y_M^*(t):=\dfrac{\partial }{\partial s}g_M^*(t,s)_{\mid s=b}\,,\end{equation}
	 we can affirm that $g_M(t,s)\geq 0$ satisfies
	 	 {\small \begin{equation}\label{Ec::pbb}
	 	 	\forall\, s\in(a,b)\quad \exists\, \eta(s)>0 \text{ such that } g_M(t,s)=g_M^1(t,s)>0 \ \forall\, t\in (b-\eta(s),b)\,
	 	 	\end{equation}} for all $M\in[0,-\lambda_2']$ and that if $M>-\lambda_2'$ the Green's function oscillates.
	 
	 So,we have already seen that $g_M(t,s)$ satisfies properties \eqref{Ec::pa},\eqref{Ec::pb}, \eqref{Ec::paa} and  \eqref{Ec::pbb} for all $M\in[0,-\lambda_2]$, and that if $M>-\lambda_2$  Green's function changes its sign on $I\times I$.
	  
	  
	 \begin{itemize}
	 	\item[Step 4.] Sign change must begin on the boundary.
	 \end{itemize}
	
	For every fixed $s\in(a,b)$, we denote $u_M^s(t)\equiv g_M(t,s)$.
	
By definition of Green's function, we have that  $T[M]\,u_M^s(t)=0$, for every $t\in I\backslash\{s\}$.

So, $T[0]\,u_M^s(t)=-M\,u_M^s(t)\leq 0$, while $u_M^s(t)\geq 0$ for $M>0$.

From the decomposition given in \eqref{e-de-T[0]}, we have that $\dfrac{1}{v_2}\dfrac{d}{dt}\left( \dfrac{{u_M^s}''}{v_1}\right) $ is a non-increasing function which, since it has a jump at $t=s$,  can vanish at most twice on $I$ . With maximal oscillation it has the shape as Figure \ref{Fig::1}.

\begin{figure}[h]
	\centering
	\includegraphics[width=0.5\textwidth]{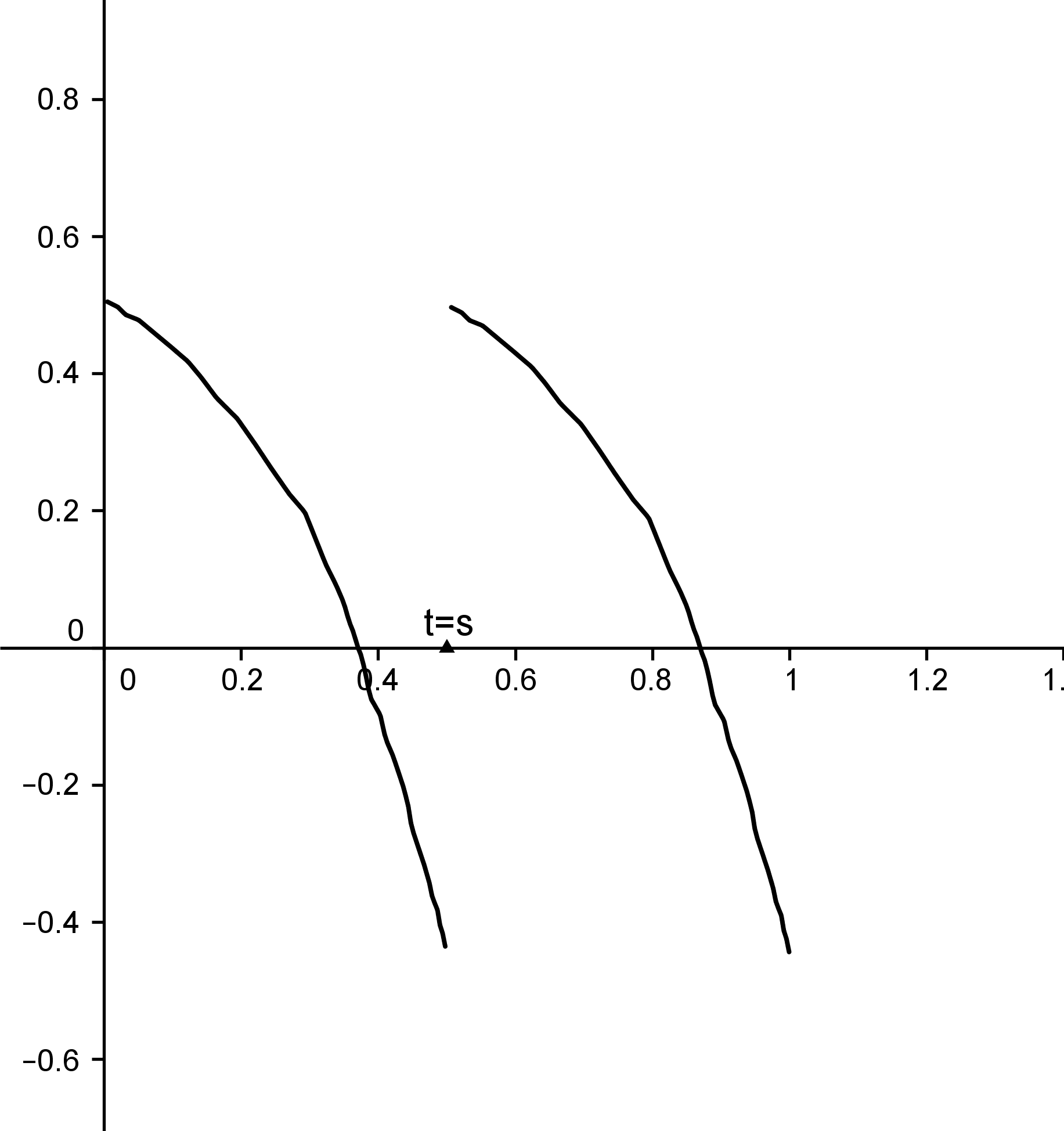}
	\caption{ \scriptsize{$\dfrac{1}{v_2}\dfrac{d}{dt}\left( \dfrac{{u_M^s}''}{v_1}\right) $ with maximal oscillation.}\label{Fig::1}}
\end{figure}
Since $v_2>0$ on $I$, $\dfrac{d}{dt}\left( \dfrac{{u_M^s}''}{v_1}\right)$ has the same sign as $\dfrac{1}{v_2}\dfrac{d}{dt}\left( \dfrac{{u_M^s}''}{v_1}\right) $. So, since $\dfrac{{u_M^s}''}{v_1}$ is a continuous function on $I$, it can have at  most four zeros, having two on the boundary. Hence, with maximal oscillation it is as the function shown in Figure \ref{Fig::2}.

\begin{figure}[h]
	\centering
	\includegraphics[width=0.5\textwidth]{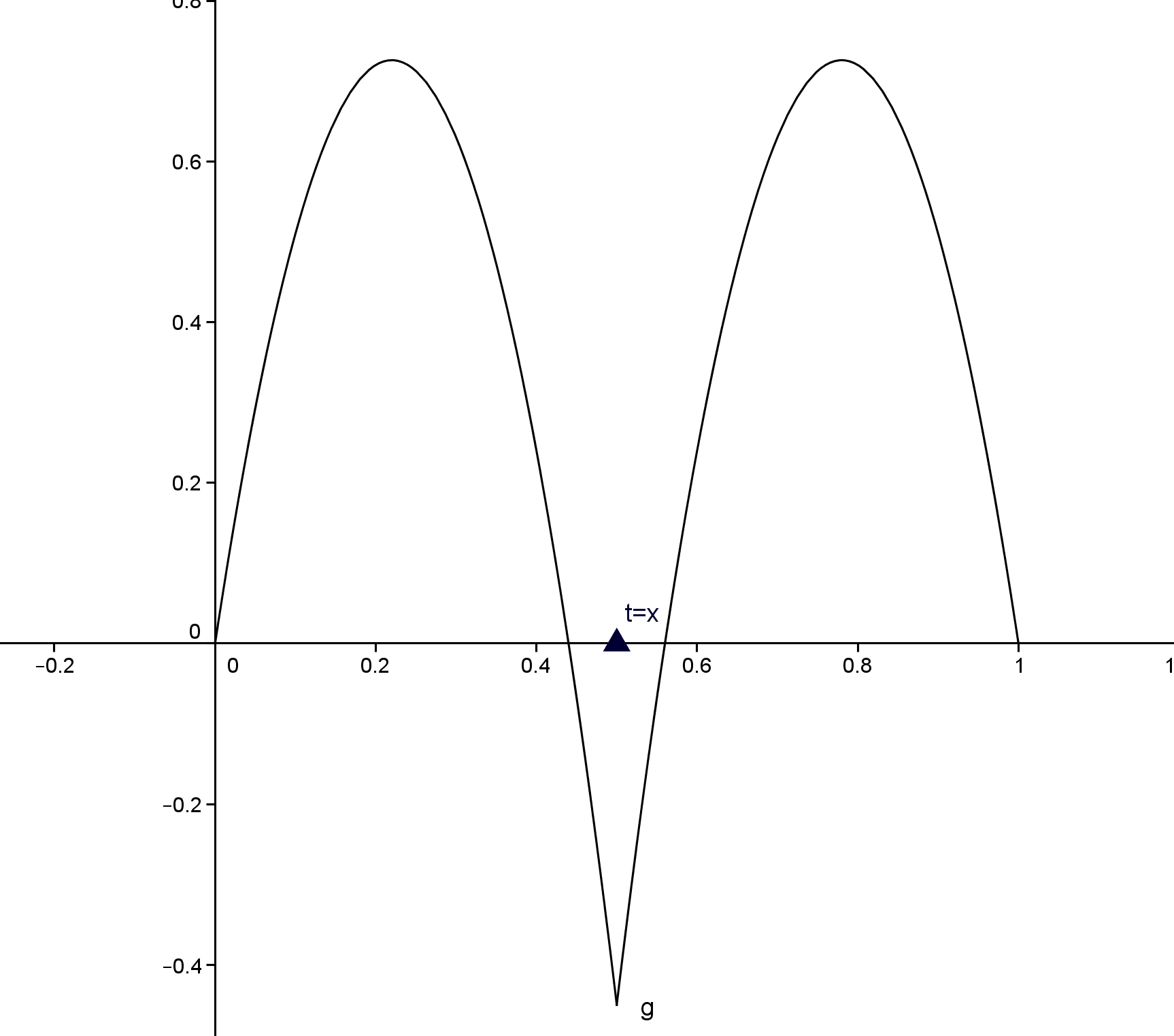}
	\caption{\scriptsize{$\dfrac{{u_M^s}''}{v_1}$ with maximal oscillation.}\label{Fig::2}}
\end{figure}

And, due to the fact that $v_1>0$ on $I$, ${u_M^s}''$  has at most two sign changes, being positive on a neighborhood of $t=a$ and $t=b$. So, ${u_M^s}'(t)$ has at most  three sign changes. However, since $u_M^s\geq 0$ and $u_M^s(a)=u_M^s(b)=0$, we can affirm that ${u_M^s}'(a)\geq 0$ and ${u_M^s}'(b)\leq 0$, which implies that ${u_M^s}'$ has at most a sign change. Then $u_M>0$ on $(a,b)$, since it vanish at $t=a$ and $t=b$.

%
%
%
Hence, since $g_M(t,s)$ is a continuous function of $M$, we conclude that the sign change has necessarily to begin on the boundary, i.e. $g_M(t,s)\geq 0$, if, and only if, $M\in(-\lambda_1,-\lambda_2]$.

Moreover, we have proved that $g_M(t,s)>0$ on $(a,b)\times(a,b)$ and that $w_M^*(t)>0$ and $y_M^*(t)<0$ on $(a,b)$. Then, using Theorem \ref{T::14} the operator $T[M]$ is strongly inverse positive if, and only if, $M\in(-\lambda_1,-\lambda_2]$.

\begin{itemize}
	\item [Step 5.] Study of the inverse negative character.
\end{itemize}

	We are going to see that in the boundary of $I\times I$ the Green's function must be nonpositive for $M\in[-\lambda_3,-\lambda_1)$. And then, we will see that this constant sign on the boundary, implies the constant sign on $(a,b)\times (a,b)$.
		
	From Lemma \ref{L::10} we know that the eigenfunctions related to $\lambda_3'$ and $\lambda_3''$ are of constant sign in $I$. Moreover, every function satisfying the boundary conditions \eqref{Ec::cfuc} on $[a,b]$ is of constant sign on $I$ for every $M\in [-\lambda_3',0]$. In particular this is true for $w_M$ defined in \eqref{Ec::wm}. The same property holds for any solution satisfying \eqref{Ec::cfud} in $[a,b]$ for all $M\in [-\lambda_3'',0]$. In particular, for $y_M$ defined in \eqref{Ec::ym}. 
		
		Since for $M=-\lambda_3'$, $w_{-\lambda_3'}(a)=w_{-\lambda_3'}'(a)=0$ and $w_{-\lambda_3'}''(a)=-1<0$, we conclude that $w_{-\lambda_3'}(t)\leq 0$ on $I$. Since it is a continuous function of $M\in[-\lambda_3',-\lambda_1]$, $w_M(t)\leq 0$, and  $g_M(t,s)$ satisfies an analogous  property to \eqref{Ec::pa} for the nonpositive case:
		 {\small \begin{equation}\label{Ec::na}
		 	\forall\, t\in(a,b)\quad \exists\, \eta(t)>0 \text{ such that } g_M(t,s)=g_M^2(t,s)<0\,, \ \forall\, s\in (a,a+\eta(t))\,.
		 	\end{equation}}
		
		Now, for $M=-\lambda_3''$, $y_{-\lambda_3''}(b)=y_{-\lambda_3''}'(b)=0$ and $y_{-\lambda_3''}''(b)=1$, so $y_{-\lambda_3''}(t)\geq0$ on $I$. Hence $y_M(t)\geq0$ for all $M\in[-\lambda_3'',-\lambda_1)$ and then $g_M(t,s)$ satisfies an analogous property to \eqref{Ec::pb} 
		 {\small \begin{equation}\label{Ec::nb}
		 	\forall \,t\in(a,b)\quad \exists\, \eta(t)>0 \text{ such that } g_M(t,s)=g_M^1(t,s)>0\,,\  \forall\, s\in (b-\eta(t),b)\,.
		 	\end{equation}}
		Studying the adjoint operator we arrive to similar conclusions for neighborhoods of $t=a$ and $t=b$ for all $M\in[-\lambda_3'',-\lambda_1)$ and $M\in[-\lambda_3',-\lambda_1)$, respectively.
		 {\small \begin{equation}\label{Ec::naa}
		 	\forall\, s\in(a,b)\quad \exists \,\eta(s)>0 \text{ such that } g_M(t,s)=g_M^2(t,s)<0\,,\ \forall\, t\in (a,a+\eta(s))\,.
		 	\end{equation}}
		 {\small \begin{equation}\label{Ec::nbb}
		 	\forall \,s\in(a,b)\quad \exists \,\eta(s)>0 \text{ such that } g_M(t,s)=g_M^1(t,s)<0\,,\ \forall\, t\in (b-\eta(s),b)\,.
		 	\end{equation}}
		
		Then $g_M(t,s)$ satisfies \eqref{Ec::na}, \eqref{Ec::nb}, \eqref{Ec::naa} and \eqref{Ec::nbb} for all $M\in[-\lambda_3,-\lambda_1)$.
	
		\vspace{0.8cm}
		
		For every $s\in(a,b)$, let  us denote  $u_M^s(t)\equiv g_M(t,s)$.
		
		Now, let us suppose that the Green's function has a zero on $(a,b)\times (a,b)$ for some $\bar{M}\in[-\lambda_3,-\lambda_1)$, hence there exists $\bar{s}\in(a,b)$, such that $u_{\bar{M}}^{\bar{s}}$ has a zero on $(a,b)$. Since $\bar{M}\in[-\lambda_3,-\lambda_1)$, $g_{\bar{M}}(t,s)$ satisfies properties \eqref{Ec::naa} and \eqref{Ec::nbb}, hence $u_{\bar{M}}^{\bar{s}}(t)<0$ on a neighborhood of $t=a$ and $t=b$. 
				
		So, under our assumption, there are two possibilities, either $u_{\bar{M}}^{\bar{s}}$ has a double zero or $u_{\bar{M}}^{\bar{s}}$ has at least two different zeros.

		\vspace{0.5cm}	
		If $u_{\bar{M}}^{\bar{s}}$  has a double zero at a point $c\in(a,b)$, we have two possibilities either $c\geq \bar{s}$ or $c\leq \bar{s}$.
		
		If $c\geq \bar{s}$, there exists an eigenvalue of $T[0]$ on the interval $[c,b]$, verifying the correspondent boundary conditions to $U_{[c,b]}$ on $[c,b]\subset I$. And this eigenvalue is less or equal than $\lambda_3''$, so this fact contradicts  Lemma \ref{L::12}.
				
		If $c\leq \bar{s}$, there exists an eigenvalue of $T_4[0]$ on $[a,c]\subset I$, verifying the boundary conditions correspondent to $V_{[a,c]}$. And, again this value is less or equal than $\lambda_3'$, which contradicts  Lemma \ref{L::12}.
		
		\vspace{0.5cm}
			If $u_{\bar{M}}^{\bar{s}}$ has two different zeros at points $c_1,c_2\in(a,b)$, such that $c_1<c_2$. Since $u_{\bar{M}}^{\bar{s}}$ is a $C^1$ function  on $I$, it must exist $d\in (c_1,c_2)$ such that ${u_{\bar{M}}^{\bar{s}}}'(d)=0$. Again, we have two possibilities either $d\leq \bar{s}$ or $d\geq \bar{s}$.
						
			If $d\leq \bar{s}$, $u\in C^4([a,d])$. From Corollary \ref{Cor::2}, we know that ${\lambda_3''}_{[a,d]}>{\lambda_3''}_{[a,b]}\equiv\lambda_3''\geq\lambda_3$, so $\bar{M}\in[-{\lambda_3''}_{[a,d]},-\lambda_1)$ and  $u_{\bar{M}}^{\bar{s}}(a)={u_{\bar{M}}^{\bar{s}}}''(a)={u_{\bar{M}}^{\bar{s}}}'(d)=0$. Hence, by applying Lemma \ref{L::10}, $u_{\bar{M}}^{\bar{s}}$ does not have any zero on $(a,d)$, but $u_{{\bar{M}}^{\bar{s}}}(c_1)=0$ and $c_1\in (a,d)$.
							
			Finally, if $d\geq \bar{s}$, $u\in C^4([d,b])$. From Corollary \ref{Cor::2}, we know that ${\lambda_3'}_{[d,b]}>{\lambda_3'}_{[a,b]}\equiv\lambda_3'\geq\lambda_3$, so $\bar{M}\in[-{\lambda_3'}_{[d,b]},-\lambda_1)$ and  ${u_{\bar{M}}^{\bar{s}}}'(d)=u_{\bar{M}}^{\bar{s}}(b)={u_{\bar{M}}^{\bar{s}}}''(b)=0$. Hence, by applying Lemma \ref{L::10}, $u_{\bar{M}}^{\bar{s}}$ does not have any zero on $(d,b)$, but $u_{\bar{M}}^{\bar{s}}(c_2)=0$ and $c_2\in (d,b)$.
							
		\vspace{0.5cm}		
	
		Hence, we arrive to a contradiction of supposing that while $M\in[-\lambda_3,-\lambda_1)$, $g_M(t,s)$ has a zero on $(a,b)\times(a,b)$. So, this together with properties \eqref{Ec::na}, \eqref{Ec::nb}, \eqref{Ec::naa} and \eqref{Ec::nbb}, tell us that $N_T\neq \emptyset$ and, moreover, that $	[-\lambda_3,-\lambda_1)\in N_T$.
		
		To see that in fact  $	[-\lambda_3,-\lambda_1)=N_T$, we do a similar argument to end of Step 1, to see that for $M<-\lambda_3$, the Green's function has to change sign necessarily.	
		
		Moreover, we have seen that if $M\in[-\lambda_3,-\lambda_1)$, then $g_M(t,s)<0$ on $(a,b)\times(a,b)$, and, furthermore, 	 $\frac{\partial }{\partial t}g_M(t,s)_{\mid t=a}<0$ and $\frac{\partial }{\partial t}g_M(t,s)_{\mid t=b}>0$ for $s\in(a,b)$. Then $T[M]$ is strongly inverse negative in $X$ if, and only if, $M\in [-\lambda_3,-\lambda_1)$.	
		\end{proof}

	From Theorem \ref{T::1} we obtain a direct corollary by using the method of upper and lower solutions, see \cite{cacisa} for details.
	
	\begin{corollary}\label{Cor::1}
		We consider the operator $T[c]\,u(t)\equiv u^{(4)}(t)+p_1(t)\,u'''(t)+p_2(t)\,u''(t)+c(t)\,u(t)$, with $c\in C(I)$, under the hypothesis of Theorem \ref{T::1}. Then,
		\begin{itemize}
			\item If $-\lambda_1<c(t)\leq -\lambda_2$ for every $t\in I$, then $T[c]$ is strongly inverse positive in $X$.
			\item If $-\lambda_3\leq c(t)<-\lambda_1$ for every $t\in I$, then $T[c]$ is strongly inverse negative in $X$.
		\end{itemize}
	\end{corollary}
	\section{Particular cases}
	In this section, in order to see the usefulness of given results, several examples where it can be applied are shown. 
		\subsection{ Operator $u^{(4)}-p\,u''$, with  $p\geq0$.}
		
		Now, let us consider the operator given by the expression $u^{(4)}-p\,u''+M\,u$, with $p\geq 0$ defined on $I$. 
		
		In \cite[Theorem 2.1]{CaSaa1} it is proved that the second order linear differential equation $u''(t)+m\,u(t)=0$ is disconjugate on $I$ if, and only if, $m\in\left( -\infty, \left( \frac{\pi}{b-a}\right) ^2\right)$. 	In particular, $u''(t)-p\,u(t)=0$ is always a disconjugate equation on $I$ for all $p>0$, so, we  can apply  Theorem \ref{T::1}. 
		
		So, in order to apply Theorem \ref{T::1} we have to obtain the eigenvalues of this operator in $X$, $X_3$, $X_1$, $U_{[a,b]}$ and $V_{[a,b]}$. 
		
		In this particular case, the eigenvalues of the operator in $X_1$ and $X_3$ are the same, and also the eigenvalues in $U_{[0,1]}$ and $V_{[0,1]}$ coincide. Indeed, if we have $u\in X_1$ (resp. $u\in U_{[0,1]}$) such that  $u^{(4)}(t)-pu''(t)+M\,u(t)=0$, then $v(t)=u(1-t)$ satisfies $v\in X_3$ (resp. $v\in V_{[0,1]}$) and $v^{(4)}(t)-pv''(t)+M\,v(t)=0$ on $I$. 
		
		We consider the general expression of the solution of equation $u^{(4)}(t)-p\,u''(t)=\lambda\, u(t)$ to obtain the different eigenvalues.
		
		So, the eigenvalues of $u^{(4)}-p\,u''$ in $X$ are given by  \[\lambda=k^4\,\left( \dfrac{\pi}{b-a}\right) ^4+k^2\,p\,\left( \dfrac{\pi}{b-a}\right)^2\,,\quad k=1,2,3,\dots\] 
		
		In particular, the least positive eigenvalue is given by $\lambda_1(p)=\left( \frac{\pi}{b-a}\right) ^4+p\,\left( \frac{\pi}{b-a}\right) ^2$.
		
		The eigenvalues of $u^{(4)}-p\,u''$  in $X_3$ are given as $-\lambda$, where $\lambda$ is a positive solution of
		\[\frac{ \tan \left(\sqrt[4]{\lambda } (b-a) \sin \left(\frac{1}{2} \cot ^{-1}\left(\frac{p}{\sqrt{4 \lambda
					-p^2}}\right)\right)\right)}{\sin \left(\frac{1}{2} \cot ^{-1}\left(\frac{p}{\sqrt{4 \lambda -p^2}}\right)\right)}=\frac{2 \tanh \left(\frac{1}{2} \sqrt[4]{\lambda } (b-a) \sqrt{\sqrt{\frac{p^2}{\lambda }}+2}\right)}{ \sqrt{\sqrt{\frac{p^2}{\lambda }}+2}}\,.\]
		
		By denoting  $m_1$ as the least positive solution of this equation,  then $\lambda_2(p)=-m_1$ is the biggest negative eigenvalue of $u^{(4)}-p\,u''$ in $X_3$.
		
		Finally, the eigenvalues of $u^{(4)}-p\,u''$ in $U_{[a,b]}$ are given as the positive solutions of 
		\[\dfrac{\tan\left( \frac{(b-a)\,\sqrt{\sqrt{p^2+4\,\lambda}-p}}{\sqrt{2}}\right) }{\sqrt{\sqrt{p^2+4\,\lambda}-p}}=\dfrac{\tanh\left( \frac{(b-a)\,\sqrt{\sqrt{p^2+4\,\lambda}+p}}{\sqrt{2}}\right) }{\sqrt{\sqrt{p^2+4\,\lambda}+p}}\,.\]
		
		Then the least positive solution of this equation, $\lambda_3(p)$, is the least positive eigenvalue of $u^{(4)}-p\,u''$ in $U_{[a,b]}$.
		
		So, now  we apply  Theorem \ref{T::1} to obtain that $u^{(4)}-p\,u''+M\,u$ is strongly inverse positive in $X$ if, and only if $M\in\left( \left( \frac{\pi}{b-a}\right) ^4+p\,\left( \frac{\pi}{b-a}\right) ^2,-\lambda_2(p)\right] $. And it is strongly inverse negative in $X$ if, and only if, $M\in\left[ -\lambda_3(p),-\left( \frac{\pi}{b-a}\right) ^4+p\,\left( \frac{\pi}{b-a}\right) ^2\right) $.
	\subsubsection{The operator $u^{(4)}$}
	As a particular case we can consider $p=0$ and $[a,b]=[0,1]$, and we have the operator $u^{(4)}(t)+M\,u(t)$ on $[0,1]$. This result has already been studied in \cite{cacisa} and \cite{Sch}, but in these cases the expression of  Green's function was needed and no relationship with spectral theory has been shown.

	In this case $\lambda_1(0)=\pi^4$ is the least positive eigenvalue of $u^{(4)}$ in $X$.
	
	Moreover, $\lambda_2(0)=-\lambda^4$, where $\lambda$ is the least positive zero of
	\[\tan \left(\frac{\lambda }{\sqrt{2}}\right)=\tanh \left(\frac{\lambda }{\sqrt{2}}\right)\,,\]
	($\lambda_2(0)\approxeq -5.55^4$) is the biggest negative eigenvalue of $u^{(4)}$ in $X_3$.
	
	Finally, $\lambda_3(0)=\lambda^4$, with $\lambda$ is  least positive root of
	
	\[\tanh (\lambda)=\tan (\lambda)\,,\]
($\lambda_3(0)\approxeq 3.927^4$) is the least positive eigenvalue of $u^{(4)}$ in $U_{[0,1]}$.

Then, we can use  Theorem \ref{T::1} to conclude that $u^{(4)}+M\,u$ is strongly inverse positive if, and only if $M\in (-\pi^4,-\lambda_2(0)]$ and $u^{(4)}+M\,u$ is strongly inverse negative if, and only if $M\in [-\lambda_3(0),-\pi^4)$. These results were obtained in \cite{cacisa}, the inverse negative character, and in \cite{Sch}, the positive one; but in our case we do not need to obtain the expression of  Green's function.

	\subsection{Operators with non constant coefficients}
	
	In this Subsection we are going to study some operator with non constant coefficients, obtaining its eigenvalues numerically.
	
	The first operator which we are going to consider is $ u^{(4)}-t^2\,u''$ on $[0,1]$. Its eigenvalues in the different spaces are obtained numerically by means of the representation of the corresponding Wronskians, and are the following
		\begin{itemize}
			\item $\lambda_1\approxeq 3.164^4$ is the least positive eigenvalue in $X$.
			\item $\lambda_2''\approxeq -5.554^4$ is the biggest negative eigenvalue in $X_3$.
			\item $\lambda_2'\approxeq -5.5716^4$ is the biggest negative eigenvalue in $X_1$.
			\item  $\lambda_3''\approxeq 3.934^4$ is the least positive eigenvalue in $V_{[0,1]}$.
			\item $\lambda_3'\approxeq 3.946^4$ is the least positive eigenvalue in $U_{[0,1]}$.
		\end{itemize}
		
		In particular,   $\lambda_2\approxeq -5.554^4$ and $\lambda_3\approxeq 3.934^4$.
		
	By means of \cite[Theorem 2.1]{CaSaa1}, we know that $u''(t)-t^2\,u(t)+M\,u(t)=0$ is a disconjugate equation on $[0,1]$ if, and only if $M\in(-\infty,-\bar \lambda)$, where $\bar \lambda \approxeq-10.16$. In particular, it is disconjugate  for $M=0$.

	 So we can apply  Theorem   \ref{T::1} to conclude that  operator $u^{(4)}-t^2 \, u''+M\,u$ is strongly inverse positive in $X$ if, and only if, $M\in(-\lambda_1,-\lambda_2]\approxeq(-3.164^4,5.554^4]$ and it is strongly inverse negative in $X$ if, and only if $M\in[-\lambda_3,-\lambda_1)\approxeq [-3.934^4,-3.164^4)$.
		\vspace{0.8cm}
		
		Now, let us study the operator $u^{(4)}+ 16\, t^2\,u''+M\,u$. Using again \cite[Theorem 2.1]{CaSaa1}, we can see that the second order linear differential equation $u''(t)+16\,t^2\,u(t)+M\,u[t]=0$ is disconjugate on $[0,1]$ if, and only if $M\in(-\infty,-{\lambda_1^2}^*)$, where ${\lambda_1^2}^*\approxeq-5.05$. In particular, it holds  for $M=0$.

		 Then, in order  to apply  Theorem \ref{T::1}  we can obtain the eigenvalues of $u^{(4)}+16\,t^2\,u''$ in the different spaces. As in the previous case we obtain them numerically by means of the study of the corresponding Wronskians.

		\begin{itemize}
			\item $\lambda_1\approxeq 2.662^4$ is the least positive eigenvalue in $X$.
			\item $\lambda_2''\approxeq -5.5334^4$ is the biggest negative eigenvalue in $X_3$.
			\item $\lambda_2'\approxeq -5.2405^4$ is the biggest negative eigenvalue in $X_1$.
			\item  $\lambda_3''\approxeq 3.555^4$ is the least positive eigenvalue in $V_{[0,1]}$.
			\item $\lambda_3'\approxeq 3.809^4$ is the least positive eigenvalue in $U_{[0,1]}$.
		\end{itemize}
		
		So, $\lambda_2\approxeq -5.2405^4$ and $\lambda_3\approxeq 3.555^4$ and we can affirm that  operator $u^{(4)}+16\,t^2\,u''+M\,u$ is strongly inverse positive in $X$ if, and only if, $M\in(-\lambda_1,-\lambda_2]\approxeq (-2.662^4,5.2405^4]$ and it is strongly inverse negative in $X$ if, and only if, $M\in[-\lambda_3,-\lambda_1)\approxeq [-3.555^4,-2.662^4)$.
		\vspace{0.8 cm}
		
		Finally, we consider  operator $u^{(4)}+2t\,u'''+2\,u''+M\,u$ on $[0,1]$.
		
		Every solution of $u''(t)+2\,t\,u'(t)+2\,u(t)=0$ satisfying $u(0)=0$ and $u'(0)=1$, follows the expression \[y(t)=e^{-t^2}\,\int_0^te^{s^2}\,ds\,,\] which does not vanish for every $t>0$. Then the second order linear  differential equation is disconjugate on every interval $[0,c]$, in particular on $[0,1]$. So we can apply  Theorem \ref{T::1}. In order to do that, we obtain numerically the eigenvalues of $u^{(4)}+2\,t\,u'''+2\,u''$ in the different spaces of definition as follows:
		
		\begin{itemize}
			\item $\lambda_1\approxeq3.079 ^4$ is the least positive eigenvalue in $X$.
			\item $\lambda_2''\approxeq -5.595^4$ is the biggest negative eigenvalue in $X_3$.
			\item $\lambda_2'\approxeq -5.606^4$ is the biggest negative eigenvalue in $X_1$.
			\item  $\lambda_3''\approxeq 3.986^4$ is the least positive eigenvalue in $V_{[0,1]}$.
			\item $\lambda_3'\approxeq3.854^4$ is the least positive eigenvalue in $U_{[0,1]}$.
		\end{itemize}
		
		Then, $\lambda_2\approxeq -5.595^4$ and $\lambda_3\approxeq3.854^4$ and we can conclude, applying Theorem \ref{T::1}, that $u^{(4)}+2\,t\,u'''+2\,u''+M\,u$ is inverse positive in $X$ if, and only if, $M\in (-\lambda_1,-\lambda_2]\approxeq(-3.079^4,5.595^4]$ and it is inverse negative in $X$ if, and only if, $M\in[-\lambda_3,-\lambda_1)\approxeq[-3.854^4,-3.079^4)$.

\end{document}